\documentclass[11pt]{article}%
\usepackage{graphicx}
\usepackage{amsmath}
\usepackage{amsfonts}
\usepackage{amssymb}%
\providecommand{\U}[1]{\protect\rule{.1in}{.1in}}
\newtheorem{theorem}{Theorem}[section]

\newtheorem{definition}[theorem]{Definition}

\newtheorem{lemma}[theorem]{Lemma}

\newtheorem{proposition}[theorem]{Proposition}
\newtheorem{remark}[theorem]{Remark}

\newenvironment{proof}[1][Proof]{\textbf{#1.} }{\hfill\rule{0.5em}{0.5em}}
{\catcode`\@=11\global\let\AddToReset=\@addtoreset
\AddToReset{equation}{section}

\AddToReset{theorem}{section}

\begin{document}

\title{Local and global estimates of solutions of Hamilton-Jacobi parabolic equation
with absorption }
\author{Marie Fran\c{c}oise BIDAUT-VERON}
\date{}
\maketitle

\begin{abstract}
Here we show new apriori estimates for the nonnegative solutions of the
equation
\[
u_{t}-\Delta u+|\nabla u|^{q}=0
\]
in $Q_{\Omega,T}=\Omega\times\left(  0,T\right)  ,$ $T\leqq\infty,$ where
$q>0,$ and $\Omega=\mathbb{R}^{N},$ or $\Omega$ is a smooth bounded domain of
$\mathbb{R}^{N}$ and $u=0$ on $\partial\Omega\times\left(  0,T\right)  .$

In case $\Omega=\mathbb{R}^{N},$ we show that any solution $u\in
C^{2,1}(Q_{\mathbb{R}^{N},T})$ of equation (\ref{un}) in $Q_{\mathbb{R}^{N}%
,T}$ (in particular \textbf{ }any weak solution if $q\leqq2),$ without
condition as $\left\vert x\right\vert \rightarrow\infty,$ satisfies the
universal estimate
\[
\left\vert \nabla u(.,t)\right\vert ^{q}\leqq\frac{1}{q-1}\frac{u(.,t)}%
{t},\qquad\text{in }Q_{\mathbb{R}^{N},T}.
\]
Moreover we prove that the growth of $u$ is limited by $C(t+t^{-1/(q-1}%
)(1+\left\vert x\right\vert ^{q^{\prime}}),$ where $C$ depends on $u.$

We also give existence properties of solutions in $Q_{\Omega,T},$ for initial
data locally integrable or unbounded measures. We give a nonuniqueness result
in case $q>2.$ Finally we show that besides the local regularizing effect of
the heat equation, $u$ satisfies a second effect of type $L_{loc}^{R}%
-L_{loc}^{\infty},$ due to the gradient term.\bigskip


\textbf{A.M.S. Subject Classification }35K15, 35K55, 35B33, 35B65, 35D30

\end{abstract}
\tableofcontents

\qquad\qquad\qquad\qquad\qquad\qquad\qquad\qquad\qquad\qquad\qquad\qquad
\qquad\qquad\qquad\qquad\qquad\qquad\qquad\qquad\qquad\qquad\qquad\qquad
\qquad\qquad\qquad\qquad\qquad\qquad\qquad\qquad\qquad\qquad\qquad
\qquad.\pagebreak

\section{Introduction\label{sec1}}

Here we consider the \textit{nonnegative} solutions of the parabolic
Hamilton-Jacobi equation%
\begin{equation}
u_{t}-\nu\Delta u+|\nabla u|^{q}=0, \label{un}%
\end{equation}
where $q>1,$ in $Q_{\Omega,T}=\Omega\times\left(  0,T\right)  ,$ where
$\Omega$ is any domain of $\mathbb{R}^{N},$ $\nu\in\left(  0,1\right]  .$ We
study the problem of apriori estimates of the \textit{nonnegative} solutions,
with possibly rough \textit{unbounded} initial data
\begin{equation}
u(x,0)=u_{0}\in\mathcal{M}^{+}(\Omega), \label{de}%
\end{equation}
where we denote by $\mathcal{M}^{+}(\Omega)$ the set of nonnegative Radon
measures in $\Omega,$ and $\mathcal{M}_{b}^{+}(\Omega)$ the subset of bounded
ones. We say that $u$ is a solution of (\ref{un}) if it satisfies (\ref{un})
in $Q_{\Omega,T}$ in the weak sense of distributions, see Section \ref{sec2}.
We say that $u$ \textit{has a trace }$u_{0}$ \textit{in} $\mathcal{M}%
^{+}(\Omega)$ if $u(.,t)$ converges to $u_{0}$ in the weak$^{\ast}$ topology
of measures:%
\begin{equation}
\lim_{t\rightarrow0}\int_{\Omega}u(.,t)\psi dx=\int_{\Omega}\psi du_{0}%
,\qquad\forall\psi\in C_{c}(\Omega). \label{mea}%
\end{equation}
Our purpose is to obtain apriori estimates valid for any solution in
$Q_{\Omega,T}=\Omega\times\left(  0,T\right)  $, without assumption on the
boundary of $\Omega,$ or for large $\left\vert x\right\vert $ if
$\Omega=\mathbb{R}^{N}.\medskip$

Fisrt recall some known results. The Cauchy problem in $Q_{\mathbb{R}^{N},T}$
\begin{equation}
(P_{\mathbb{R}^{N},T})\left\{
\begin{array}
[c]{l}%
u_{t}-\nu\Delta u+|\nabla u|^{q}=0,\quad\text{in}\hspace{0.05in}%
Q_{\mathbb{R}^{N},T},\\
u(x,0)=u_{0}\quad\quad\text{in}\hspace{0.05in}\mathbb{R}^{N},
\end{array}
\right.  \label{cau}%
\end{equation}
is the object of a rich literature, see among them \cite{AmBA},\cite{BeLa99},
\cite{BeBALa}, \cite{BASoWe}, \cite{SoZh},\cite{BiDao1}, \cite{BiDao2}, and
references therein. The first studies concern \textit{classical} solutions,
that means $u\in C^{2,1}(Q_{\mathbb{R}^{N},T}),$ with \textit{smooth bounded
initial data} $u_{0}\in C_{b}^{2}\left(  \mathbb{R}^{N}\right)  $: there a
unique global solution such that
\[
\left\Vert u(.,t)\right\Vert _{L^{\infty}(\mathbb{R}^{N})}\leqq\left\Vert
u_{0}\right\Vert _{L^{\infty}(\mathbb{R}^{N})},\text{ and }\left\Vert \nabla
u(.,t)\right\Vert _{L^{\infty}(\mathbb{R}^{N})}\leqq\left\Vert \nabla
u_{0}\right\Vert _{L^{\infty}(\mathbb{R}^{N})},\qquad\text{in }Q_{\mathbb{R}%
^{N},T},
\]
see \cite{AmBA}. Then universal apriori estimates of the gradient are obtained
\textit{for this solution}, by using the Bersnstein technique, which consists
in computing the equation satisfied by $|\nabla u|^{2}:$ first from
\cite{Li},
\[
\left\Vert \nabla u(.,t)\right\Vert _{L^{\infty}(\mathbb{R}^{N})}^{q}%
\leqq\frac{\left\Vert u_{0}\right\Vert _{L^{\infty}(\mathbb{R}^{N})}}{t},
\]
in $Q_{\mathbb{R}^{N},T},,$ then from \cite{BeLa99},%
\begin{equation}
\left\vert \nabla u(.,t)\right\vert ^{q}\leqq\frac{1}{q-1}\frac{u(.,t)}{t},
\label{verse}%
\end{equation}%
\begin{equation}
\Vert\nabla(u^{\frac{q-1}{q}})(.,t)\Vert_{L^{\infty}(\mathbb{R}^{N})}\leqq
Ct^{-1/2}\Vert u_{0}\Vert_{L^{\infty}(\mathbb{R}^{N})}^{\frac{q-1}{q}},\qquad
C=C(N,q,\nu). \label{car}%
\end{equation}
Existence and uniqueness was extended to any $u_{0}\in C_{b}\left(
\mathbb{R}^{N}\right)  $ in \cite{GiGuKe}; then the estimates (\ref{car}) and
(\ref{verse}) are still valid, see \cite{BeBALa}. In case of nonnegative rough
initial data $u_{0}\in L^{R}\left(  \mathbb{R}^{N}\right)  ,$ $R\geqq1,$ or
$u_{0}\in\mathcal{M}_{b}^{+}(\mathbb{R}^{N}),$ the problem was studied in a
semi-group formulation \cite{BeLa99}, \cite{BASoWe}, \cite{SoZh}, then in the
larger class of weak solutions in \cite{BiDao1}, \cite{BiDao2}. Recall that
two critical values appear: $q=2,$ where the equation can be reduced to the
heat equation, and
\[
q_{\ast}=\frac{N+2}{N+1}.
\]
Indeed the Cauchy problem with initial value $u_{0}=\kappa\delta_{0},$ where
$\delta_{0}$ is the Dirac mass at $0$ and $\kappa>0,$ has a weak solution
$u^{\kappa}$ if and only if $q<q_{\ast},$ see \cite{BeLa99}, \cite{BiDao1}.
Moreover as $\kappa\rightarrow\infty,$ $(u^{\kappa})$ converges to a unique
very singular solution $Y,$ see \cite{QW}, \cite{BeLa01}, \cite{BeKoLa},
\cite{BiDao1}. And $Y(x,t)=t^{-a/2}F(\left\vert x\right\vert /\sqrt{t}),\quad
$where
\begin{equation}
a=\frac{2-q}{q-1}, \label{vala}%
\end{equation}
and $F$ is bounded and has an exponential decay at infinity.\medskip

In \cite[Theorem 2.2]{BiDao2} it is shown that for any $R\geqq1$ global
regularizing $L^{R}$-$L^{\infty}$ properties of two types hold for the Cauchy
problem in $Q_{\mathbb{R}^{N},T}$ : one due to the heat operator:
\begin{equation}
\Vert u(.,t)\Vert_{L^{\infty}(\mathbb{R}^{N})}\leqq Ct^{-\frac{N}{2R}}\Vert
u_{0}\Vert_{L^{R}(\mathbb{R}^{N})},\qquad C=C(N,R,\nu), \label{efa}%
\end{equation}
and the other due to the gradient term, independent of $\nu$ ($\nu>0$)$:$%
\begin{equation}
\Vert u(.,t)\Vert_{L^{\infty}(\mathbb{R}^{N})}\leqq Ct^{-\frac{N}{qR+N(q-1)}%
}\Vert u_{0}\Vert_{L^{R}(\mathbb{R}^{N})}^{\frac{qR}{qR+N(q-1)}},\qquad
C=C(N,q,R). \label{efb}%
\end{equation}

A great part of the results has been extended to the Dirichlet problem in a
bounded domain $\Omega:$%
\begin{equation}
(P_{\Omega,T})\left\{
\begin{array}
[c]{l}%
u_{t}-\Delta u+|\nabla u|^{q}=0,\quad\text{in}\hspace{0.05in}Q_{\Omega,T},\\
u=0,\qquad\qquad\text{on}\hspace{0.05in}\partial\Omega\times(0,T),\\
u(x,0)=u_{0},
\end{array}
\right.  \label{1.1}%
\end{equation}
where $u_{0}\in\mathcal{M}_{b}^{+}(\Omega)$, and $u(.,t)$ converges to $u_{0}$
weakly in $\mathcal{M}_{b}^{+}(\Omega)$, see \cite{BeDa}, \cite{SoZh},
\cite{BiDao1}, \cite{BiDao2}. Universal estimates are given in \cite{CLS}, see
also \cite{BiDao1}. Note that (\ref{verse}) cannot hold, since it contradicts
the H\"{o}pf Lemma. \medskip

Finally local estimates in any domain $\Omega$ were proved in \cite{SoZh}: for
any classical solution $u$ in $Q_{\Omega,T}$ and any ball $B(x_{0}%
,2\eta)\subset\Omega,$ there holds in $Q_{B(x_{0},\eta),T}$%
\begin{equation}
\left\vert \nabla u\right\vert (.,t)\leqq C(t^{-\frac{1}{q}}+\eta^{-1}%
+\eta^{-\frac{1}{q-1}})(1+u(.,t)),\qquad C=C(N,q,\nu). \label{soup}%
\end{equation}

\subsection{Main results}

In Section \ref{sec3} we give \textit{local integral estimates} of the
solutions \textit{in terms of the initial data, and a first regularizing
effect, }local version of (\ref{efa}), see Theorem \ref{local}.\textit{ }

\begin{theorem}
Let $q>1.$ Let $u$ be any nonnegative weak solution of equation (\ref{un}) in
$Q_{\Omega,T}$, and let $B(x_{0},2\eta)\subset\subset\Omega$ such that $u$ has
a trace $u_{0}\in L_{loc}^{R}(\Omega),$ $R\geqq1$ and $u\in C(\left[
0,T\right)  ;L_{loc}^{R}(\Omega)).$ Then for any $0<t\leqq\tau<T,$%
\[
\sup_{x\in B(x_{0},\eta/2)}u(x,t)\leqq Ct^{-\frac{N}{2R}}(t+\left\Vert
u_{0}\right\Vert _{L^{R}(B(x_{0},\eta)}),\qquad C=C(N,q,\nu,R,\eta,\tau).
\]
If $R=1,$ the estimate remains true when $u_{0}\in\mathcal{M}^{+}(\Omega)$
(with $\left\Vert u_{0}\right\Vert _{L^{1}(B(x_{0},\eta)}$ replaced by
$\int_{B(x_{0},\eta)}du_{0}$).$\medskip$
\end{theorem}

In Section \ref{sec4}, we give \textit{global estimates} of the solutions of
(\ref{un}) in $Q_{\mathbb{R}^{N},T},$ and this is our main result. We show
that \textit{the universal estimate (\ref{verse}) in }$\mathbb{R}^{N}$\textit{
holds without assuming that the solutions are initially bounded}:

\begin{theorem}
\label{fund}Let $q>1.$ Let $u$ be any classical solution, in particular
\textbf{ any weak solution} if $q\leqq2,$ of equation (\ref{un}) in
$Q_{\mathbb{R}^{N},T}.$ Then
\begin{equation}
\left\vert \nabla u(.,t)\right\vert ^{q}\leqq\frac{1}{q-1}\frac{u(.,t)}%
{t},\qquad\text{in }Q_{\mathbb{R}^{N},T}. \label{versa}%
\end{equation}

\end{theorem}

And we prove that \textit{the growth of the solutions is limited, in
}$\left\vert x\right\vert ^{q^{\prime}}$\textit{as }$\left\vert x\right\vert
\rightarrow\infty$\textit{ and in}\textbf{ }$t^{-1/(q-1)}$\textbf{
}\textit{as}\textbf{ }$t\rightarrow0$:

\begin{theorem}
\label{growth}Let $q>1.$ Let $u$ be any classical solution, in particular
\textbf{ any weak solution} if $q\leqq2,$ of equation (\ref{un}) in
$Q_{\mathbb{R}^{N},T},$ such that there exists a ball $B(x_{0},2\eta)$ such
that $u$ has a trace $u_{0}\in\mathcal{M}^{+}((B(x_{0},2\eta)).$ Then%
\begin{equation}
u(x,t)\leqq C(q)t^{-\frac{1}{q-1}}\left\vert x-x_{0}\right\vert ^{q^{\prime}%
}+C(t^{-\frac{1}{q-1}}+t+\int_{B(x_{0},\eta)}du_{0}),\qquad C=C(N,q,\eta).
\label{esta}%
\end{equation}

\end{theorem}

In \cite{BiDao3}, we show that there exist solutions with precisely this type
of behaviour of order $t^{-1/(q-1)}\left\vert x\right\vert ^{q^{\prime}}$ as
$\left\vert x\right\vert \rightarrow\infty$\textbf{ }or \textbf{
}$t\rightarrow0$. Moreover we prove that the condition on the trace is always
satisfied for $q<q_{\ast}.\medskip$

In Section \ref{sec5} we complete the study by giving \textit{existence
results} with only \textit{local assumptions on }$u_{0}$, extending some
results of \cite{BeBALa} where $u_{0}$ is continuous, and \cite{BASoWe},
\cite{BiDao2}, where the assumptions are global:

\begin{theorem}
\label{exim}Let $\Omega=\mathbb{R}^{N}$ (resp. $\Omega$ bounded).\medskip

(i) If $1<q<q_{\ast}$, then for any $u_{0}\in\mathcal{M}^{+}\left(
\mathbb{R}^{N}\right)  $ (resp. $\mathcal{M}^{+}\left(  \Omega\right)  $),
there exists a weak solution $u$ of equation (\ref{un}) (resp. of
$(D_{\Omega,T})$) with trace $u_{0}$.\medskip

(ii) If $q_{\ast}\leqq q\leqq2,$ then existence still holds for any
nonnegative $u_{0}\in L_{loc}^{1}\left(  \mathbb{R}^{N}\right)  $ (resp.
$L_{loc}^{1}\left(  \Omega\right)  $). And then $u\in C(\left[  0,T\right)
;L_{loc}^{1}\left(  \mathbb{R}^{N}\right)  $ (resp. $u\in C(\left[
0,T\right)  ;L_{loc}^{1}\left(  \Omega\right)  ).$\medskip

(iii) $If$ $q>2,$ existence holds for any nonnegative $u_{0}\in L_{loc}%
^{1}\left(  \mathbb{R}^{N}\right)  $ (resp. $L_{loc}^{1}\left(  \Omega\right)
$) which is limit of a nondecreasing sequence of continuous
functions.$\medskip$
\end{theorem}

Moreover we give a result of \textit{nonuniqueness} of weak solutions in case
$q>2:$

\begin{theorem}
\label{twosol}Assume that $q>2,$ $N\geq2.$Then the Cauchy problem
$(P_{\mathbb{R}^{N},\infty})$ with initial data
\[
\tilde{U}(x)=\tilde{C}\left\vert x\right\vert ^{\left\vert a\right\vert }\in
C\left(  \mathbb{R}^{N}\right)  ,\qquad\tilde{C}=\frac{q-1}{q-2}%
(\frac{(N-1)q-N)}{q-1})^{\frac{1}{q-1}},
\]
admits at least two weak solutions: the stationary solution $\tilde{U},$ and a
radial self-similar solution of the form
\begin{equation}
U_{\tilde{C}}(x,t)=t^{\left\vert a\right\vert /2}f(\left\vert x\right\vert
/\sqrt{t}), \label{fom}%
\end{equation}
where $f$ is increasing on $\left[  0,\infty\right)  ,$ $f(0)>0,$ and
$\lim_{\eta\rightarrow\infty}\eta^{-\left\vert a\right\vert /2}f(\eta
)=\tilde{C}.\medskip$
\end{theorem}

Finally we give in Section \ref{sec6} a second type of regularizing effects
giving a local version of (\ref{efb}).

\begin{theorem}
\label{effects}Let $q>1,$ and let $u$ be any nonnegative classical solution
(resp. any weak solution if $q\leqq2)$ of equation (\ref{un}) in $Q_{\Omega
,T}$, and let $B(x_{0},2\eta)\subset\Omega$. Assume that $u_{0}\in L_{loc}%
^{R}(\Omega)$ for some $R\geqq1,$ $R>q-1,$ and $u\in C(\left[  0,T\right)
;L_{loc}^{R}(\Omega)).$ Then for any $\varepsilon>0,$ and for any $\tau
\in\left(  0,T\right)  ,$ then there exists $C=C(N,q,R,\eta,\varepsilon,\tau)$
such that
\begin{equation}
\text{sup}_{B_{\eta/2}}u(.,t)\leqq Ct^{-\frac{N}{qR+N(q-1)}}(t+\left\Vert
u_{0}\right\Vert _{L^{R}(B_{\eta})})^{\frac{Rq}{qR+N(q-1)}}+Ct^{\frac
{1-\varepsilon}{R+1-q}}\left\Vert u_{0}\right\Vert _{L^{R}(B_{\eta})}%
^{\frac{R}{R+1-q}}. \label{epsi}%
\end{equation}
If $q<2,$ the estimates for $R=1$ are also valid when $u$ has a trace
$u_{0}\in\mathcal{M}^{+}(\Omega),$ with $\left\Vert u_{0}\right\Vert
_{L^{1}(B_{\eta})}$ replaced by $\int_{B_{\eta}}du_{0}.\medskip$
\end{theorem}

In conclusion, note that a part of our results could be extended to more
general quasilinear operators, for example to the case of equation involving
the $p$-Laplace operator
\[
u_{t}-\nu\Delta_{p}u+|\nabla u|^{q}=0
\]
with $p>1,$ following the results of \cite{BiDao2}, \cite{BarLa}, \cite{IaLa},
\cite{FoSoVe}.

\section{Classical and weak solutions\label{sec2}}

We set $Q_{\Omega,s,\tau}=\Omega\times\left(  s,\tau\right)  ,$ for any
$0\leqq s<\tau\leqq\infty,$ thus $Q_{\Omega,T}=Q_{\Omega,0,T}.\medskip$

\begin{definition}
\label{defw}Let $q>1$ and $\Omega$ be any domain of $\mathbb{R}^{N}.$ We say
that a nonnegative function $u$ is a \textbf{classical} solution of (\ref{un})
in $Q_{\Omega,T}$ if $u\in C^{2,1}(Q_{\Omega,T})$. We say that $u$ is a
\textbf{weak solution} (resp. weak subsolution) of (\ref{un}) in $Q_{\Omega
,T},$ if $u\in C((0,T);L_{loc}^{1}(Q_{\Omega,T}))\cap L_{loc}^{1}%
((0,T);W_{loc}^{1,1}\left(  \Omega\right)  ),$ $|\nabla u|^{q}\in L_{loc}%
^{1}(Q_{\Omega,T})$ and $u$ satisfies (\ref{un}) in the distribution sense:
\begin{equation}
\int_{0}^{T}\int_{\Omega}(-u\varphi_{t}-\nu u\Delta\varphi+|\nabla
u|^{q}\varphi)=0,\quad\forall\varphi\in\mathcal{D}(Q_{\Omega,T}), \label{for}%
\end{equation}
(resp.
\begin{equation}
\int_{0}^{T}\int_{\Omega}(-u\varphi_{t}-\nu u\Delta\varphi+|\nabla
u|^{q}\varphi)\leqq0,\quad\forall\varphi\in\mathcal{D}^{+}(Q_{\Omega
,T}).\text{)} \label{fors}%
\end{equation}
And then for any $0<s<t<T,$ and any $\varphi\in C^{1}((0,T),C_{c}^{1}%
(\Omega)),$
\begin{equation}
\int_{\Omega}(u\varphi)(.,t)-\int_{\Omega}(u\varphi)(.,s)+\int_{s}^{t}%
\int_{\Omega}(-u\varphi_{t}+\nu\nabla u.\nabla\varphi+|\nabla u|^{q}%
\varphi)=0\text{ (resp.}\leqq0). \label{fort}%
\end{equation}

\end{definition}

\begin{remark}
\label{subreg} Any weak subsolution $u$ is locally bounded in $Q_{\Omega,T}$.
Indeed, since $u$ is $\nu$-subcaloric, there holds for any ball $B(x_{0}%
,\rho)\subset\subset\Omega$ and any $\rho^{2}\leqq t<T,$
\begin{equation}
\sup_{B(x_{0},\frac{\rho}{2})\times\left[  t-\frac{\rho^{2}}{4},t\right]
}u\leqq C(N,\nu)\rho^{-(N+2)}\int_{t-\frac{\rho^{2}}{2}}^{t}\int_{B(x_{0}%
,\rho)}u. \label{nusc}%
\end{equation}
Any nonnegative\textbf{ }function $u\in L_{loc}^{1}(Q_{\Omega,T}),$ such that
$|\nabla u|^{q}\in L_{loc}^{1}(Q_{\Omega,T}),$ and $u$ satisfies (\ref{for}),
is a weak solution and $\left\vert \nabla u\right\vert \in L_{loc}%
^{2}(Q_{\Omega,T})),u\in C((0,T);L_{loc}^{s}(Q_{\Omega,T})),\forall s\geqq1,$
see \cite[Lemma 2.4]{BiDao1}.$\medskip$
\end{remark}

Next we recall the regularity of the weak solutions of (\ref{un}) for
$q\leqq2,$ see \cite[Theorem 2.9]{BiDao1}, \cite[Corollary 5.14]{BiDao2}:

\begin{theorem}
\label{gul} Let $1<q\leqq2$. Let $\Omega$ be any domain in $\mathbb{R}^{N}$.
Let $u$ be any weak nonnegative solution of (\ref{un}) in $Q_{\Omega,T}$. Then
$u\in C_{loc}^{2+\gamma,1+\gamma/2}(Q_{\Omega,T})$ for some $\gamma\in\left(
0,1\right)  ,$ and for any smooth domains $\omega\subset\subset\omega^{\prime
}\subset\subset\Omega,$ and $0<s<\tau<T,$ $\left\Vert u\right\Vert
_{C^{2+\gamma,1+\gamma/2}(Q_{\omega,s,\tau})}$ is bounded in terms of
$\left\Vert u\right\Vert _{L^{\infty}(Q_{\omega^{\prime},s/2,\tau})}.$ Thus
for any sequence $(u_{n})$ of nonnegative weak solutions of equation
(\ref{un}) in $Q_{\Omega,T},$ uniformly locally bounded, one can extract a
subsequence converging in $C_{loc}^{2,1}(Q_{\Omega,T})$ to a weak solution $u$
of (\ref{un}) in $Q_{\Omega,T}.\medskip$
\end{theorem}

\begin{remark}
\label{gil}Let $q>1.$ From the estimates (\ref{soup}), for any sequence of
classical nonnegative solutions $\left(  u_{n}\right)  $ of (\ref{un}) in
$Q_{\Omega,T},$ uniformly bounded in $L_{loc}^{\infty}(Q_{\Omega,T}),$ one can
extract a subsequence converging in $C_{loc}^{2,1}(Q_{\mathbb{R}^{N},T})$ to a
classical solution $u$ of (\ref{un}).$\medskip$
\end{remark}

\begin{remark}
\label{trac}Let us mention some results of concerning the trace, valid for any
$q>1$, see \cite[Lemma 2.14]{BiDao1}. Let $u$ be any nonnegative weak solution
$u$ of (\ref{un}) in $Q_{\Omega,T}$. Then $u$ has a trace $u_{0}$ in
$\mathcal{M}^{+}(\Omega)$ if and only if $u\in L_{loc}^{\infty}{(}\left[
0,T\right)  {;L_{loc}^{1}(}\Omega)),$ and if and only if $\left\vert \nabla
u\right\vert ^{q}\in L_{loc}^{1}(\Omega\times\left[  0,T\right)  ).$\ And then
for any $t\in(0,T),$ and any $\varphi\in C_{c}^{1}(\Omega\times\left[
0,T\right)  )$, and any $\zeta\in C_{c}^{1}(\Omega),$
\begin{equation}
\int_{\Omega}u(.,t)\varphi dx+\int_{0}^{t}\int_{\Omega}(-u\varphi_{t}%
+\nu\nabla u.\nabla\varphi+\left\vert \nabla u\right\vert ^{q}\varphi
)=\int_{\Omega}\varphi(.,0)du_{0}, \label{hou}%
\end{equation}%
\begin{equation}
\int_{\Omega}u(.,t)\zeta+\int_{0}^{t}\int_{\Omega}(\nu\nabla u.\nabla
\zeta+\left\vert \nabla u\right\vert ^{q}\zeta)=\int_{\Omega}\zeta du_{0}.
\label{bou}%
\end{equation}
If $u_{0}\in{L_{loc}^{1}(}\Omega),$ then $u\in C{(}\left[  0,T\right)
{;L_{loc}^{1}(}\Omega)).$
\end{remark}

Finally we consider the Dirichlet problem in a smooth bounded domain $\Omega
$:
\begin{equation}
(D_{\Omega,T})\left\{
\begin{array}
[c]{l}%
u_{t}-\Delta u+|\nabla u|^{q}=0,\quad\text{in}\hspace{0.05in}Q_{\Omega,T},\\
u=0,\quad\text{on}\hspace{0.05in}\partial\Omega\times(0,T).
\end{array}
\right.  \label{diri}%
\end{equation}

\begin{definition}
We say that a function $u$ is a \textbf{weak solution of }$(D_{\Omega,T})$ if
it is a weak solution of equation (\ref{un}) such that $u\in C((0,T);L^{1}%
\left(  \Omega\right)  )\cap L_{loc}^{1}((0,T);W_{0}^{1,1}\left(
\Omega\right)  ),$ and $|\nabla u|^{q}\in L_{loc}^{1}((0,T);L^{1}\left(
\Omega\right)  ).$ We say that $u$ is a \textbf{classical} solution of
$(D_{\Omega,T})$ if $u\in C^{2,1}(Q_{\Omega,T})\cap C^{1,0}\left(
\overline{\Omega}\times\left(  0,T\right)  \right)  .$
\end{definition}

\section{Local integral properties and first regularizing effect\label{sec3}}

\subsection{Local integral properties}

\begin{lemma}
\label{int}Let $\Omega$ be any domain in $\mathbb{R}^{N}$, $q>1,R\geqq1.$ Let
$u$ be any nonnegative weak subsolution of equation (\ref{un}) in
$Q_{\Omega,T},$ such that $u\in C{((}0,T){;L_{loc}^{R}(}\Omega)).$ Let $\xi
\in$ $C^{1}((0,T);C_{c}^{1}(\Omega)),$ with values in $\left[  0,1\right]  .$
Let $\lambda>1.$ Then there exists $C=C(q,R,\lambda),$ such that, for any
$0<s<t\leqq\tau<T,$%
\begin{align}
&
{\textstyle\int_{\Omega}}
u^{R}(.,t)\xi^{\lambda}+\frac{1}{2}\int_{s}^{\tau}\int_{\Omega}u^{R-1}|\nabla
u|^{q}\xi^{\lambda}+\nu\frac{R-1}{2}\int_{s}^{\tau}\int_{\Omega}u^{R-2}|\nabla
u|^{2}\xi^{\lambda}\nonumber\\
&  \leqq%
{\textstyle\int_{\Omega}}
u^{R}(.,s)\xi^{\lambda}+\lambda R\int_{s}^{t}\int_{\Omega}u^{R}\xi^{\lambda
-1}\left\vert \xi_{t}\right\vert +C\int_{s}^{t}\int_{\Omega}u^{R-1}%
\xi^{\lambda-q^{\prime}}|\nabla\xi|^{q^{\prime}}. \label{ruc}%
\end{align}

\end{lemma}

\begin{proof}
(i) Let $R=1.$ Taking $\varphi=\xi^{\lambda}$ in (\ref{fort}), we obtain,
since $\nu\leqq1,$
\begin{align*}
&  \int_{\Omega}u(.,t)\xi^{\lambda}+\int_{s}^{t}\int_{\Omega}|\nabla u|^{q}%
\xi^{\lambda}\leqq\int_{\Omega}u(s,.)\xi^{\lambda}+\lambda\int_{s}^{t}%
\int_{\Omega}\xi^{\lambda-1}u\xi_{t}+\lambda\nu\int_{s}^{t}\int_{\Omega}%
\xi^{\lambda-1}\nabla u.\nabla\xi\\
&  \leqq\int_{\Omega}u(.,s)\xi^{\lambda}+\lambda\int_{s}^{t}\int_{\Omega}%
\xi^{\lambda-1}u\left\vert \xi_{t}\right\vert +\frac{1}{2}\int_{s}^{t}%
\int_{\Omega}|\nabla u|^{q}\xi^{q^{\prime}}+C(q,\lambda)\int_{s}^{t}%
\int_{\Omega}\xi^{\lambda-q^{\prime}}|\nabla\xi|^{q^{\prime}},
\end{align*}
hence (\ref{ruc}) follows.

(ii) Next assume $R>1.$ Consider $u_{\delta,n}=((u+\delta)\ast\varphi_{n})$,
where ($\varphi_{n}$) is a sequence of mollifiers, and $\delta>0.$ Then by
convexity, $u_{\delta,n}$ is also a subsolution of (\ref{un}):%
\[
(u_{\delta,n})_{t}-\nu\Delta u_{\delta,n}+|\nabla u_{\delta,n}|^{q}\leqq0.
\]
Multiplying by $u_{\delta,n}^{R-1}\xi^{\lambda}$ and integrating between $s$
and $t,$ and going to the limit as $\delta\rightarrow0$ and $n\rightarrow
\infty,$ see \cite{BiDao2}, we get with different constants $C=(N,q,R,\lambda
),$ independent of $\nu,$
\begin{align*}
&  \frac{1}{R}\int_{\Omega}u^{R}(.,t)\xi^{\lambda}+\nu(R-1)\int_{s}^{t}%
\int_{\Omega}u^{R-2}|\nabla u|^{2}\xi^{\lambda}+\int_{s}^{t}\int_{\Omega
}u^{R-1}|\nabla u|^{q}\xi^{\lambda}\\
&  \leqq\frac{1}{R}\int_{\Omega}u^{R}(.,s)\xi^{\lambda}+\lambda\int_{s}%
^{t}\int_{B_{\rho}}\xi^{\lambda-1}u^{R}\left\vert \xi_{t}\right\vert
+\lambda\nu\int_{\theta}^{t}\int_{\Omega}u^{R-1}|\nabla u|\left\vert \nabla
\xi\right\vert \xi^{\lambda-1}\\
&  \leqq\frac{1}{R}\int_{\Omega}u^{R}(.,s)\xi^{\lambda}+\lambda\int_{s}%
^{t}\int_{B_{\rho}}\xi^{\lambda-1}u^{R}\left\vert \xi_{t}\right\vert \\
&  +\frac{1}{2}\int_{s}^{\tau}\int_{\Omega}u^{R-1}|\nabla u|^{q}\xi^{\lambda
}+C(\lambda,R)\int_{s}^{t}\int_{\Omega}u^{R-1}\xi^{\lambda-q^{\prime}}%
|\nabla\xi|^{q^{\prime}},
\end{align*}
and (\ref{ruc}) follows again.\medskip
\end{proof}

Then we give local integral estimates of $u(.,t)$ in terms of the initial data:

\begin{lemma}
\label{cor}Let $q>1.$ Let $\eta>0.$ Let $u$ be any nonnegative weak solution
of equation (\ref{un}) in $Q_{\Omega,T}$, with trace $u_{0}\in\mathcal{M}%
^{+}(\Omega),$ and let $B(x_{0},2\eta)\subset\subset\Omega$. Then for any
$t\in\left(  0,T\right)  $,
\begin{equation}
\int_{B(x_{0},\eta)}u(x,t)\leqq C(N,q)\eta^{N-q^{\prime}}t+\int_{B(x_{0}%
,2\eta)}du_{0}. \label{zof}%
\end{equation}
Moreover if $u_{0}\in L_{loc}^{R}(\Omega)$ $(R>1),$ and $u\in C(\left[
0,T\right)  ;L_{loc}^{R}(\Omega)),$ then
\begin{equation}
\left\Vert u(.,t)\right\Vert _{L^{R}(B(x_{0},\eta))}\leqq C(N,q,R)\eta
^{\frac{N}{R}-q^{\prime}}t+\left\Vert u_{0}\right\Vert _{L^{R}(B(x_{0}%
,2\eta))}. \label{locint}%
\end{equation}
If $u\in C(\overline{B(x_{0},2\eta)}\times\left[  0,T\right)  ),$ then
\begin{equation}
\left\Vert u(.,t)\right\Vert _{L^{\infty}(B(x_{0},\eta))}\leqq C(N,q)\eta
^{-q^{\prime}}t+\left\Vert u_{0}\right\Vert _{L^{\infty}(B(x_{0},2\eta))}.
\label{locinf}%
\end{equation}

\end{lemma}

\begin{proof}
We can assume that $0\in\Omega$ and $x_{0}=0.$ We take $\xi\in C_{c}%
^{1}(\Omega),$ independent of $t,$ with values in $\left[  0,1\right]  ,$ and
$R=1$ in (\ref{ruc}), $\lambda=q^{\prime}$. Then for any $0<s<t<T,$%
\[
\int_{\Omega}u(.,t)\xi^{q^{\prime}}+\frac{1}{2}\int_{s}^{t}\int_{\Omega
}|\nabla u|^{q}\xi^{q^{\prime}}\leqq\int_{\Omega}u(.,s)\xi^{q^{\prime}%
}+C(q)\int_{s}^{t}\int_{\Omega}|\nabla\xi|^{q^{\prime}}\leqq\int_{\Omega
}u(.,s)\xi^{q^{\prime}}+C(q)t\int_{\Omega}|\nabla\xi|^{q^{\prime}}.
\]
Hence as $s\rightarrow0,$ we get
\begin{equation}
\int_{\Omega}u(.,t)\xi^{q^{\prime}}+\frac{1}{2}\int_{0}^{t}\int_{\Omega
}|\nabla u|^{q}\xi^{q^{\prime}}\leqq C(q)t\int_{\Omega}|\nabla\xi|^{q^{\prime
}}+\int_{\Omega}\xi^{q^{\prime}}du_{0}. \label{zif}%
\end{equation}
Then taking $\xi=1$ in $B_{\eta}$ with support in $B_{2\eta}$ and $|\nabla
\xi|\leqq C_{0}(N)/\eta,$
\begin{equation}
\int_{B_{\eta}}u(x,t)\leqq C(N,q)\eta^{N-q^{\prime}}t+\int_{B_{2\eta}}%
\xi^{q^{\prime}}du_{0}, \label{cla}%
\end{equation}
hence we get (\ref{zof}). Next assume $u_{0}\in L_{loc}^{R}(\Omega)$ $(R>1),$
and $u\in C(\left[  0,T\right)  ;L_{loc}^{R}(\Omega)).$ Then from (\ref{ruc}),
for any $0<s<t\leqq\tau<T,$ we find,
\begin{align*}%
{\textstyle\int_{\Omega}}
u^{R}(.,t)\xi^{\lambda}+\frac{1}{2}\int_{s}^{\tau}\int_{\Omega}u^{R-1}|\nabla
u|^{q}\xi^{\lambda}  &  \leqq%
{\textstyle\int_{\Omega}}
u^{R}(.,s)\xi^{\lambda}+\int_{s}^{t}\int_{\Omega}u^{R-1}\xi^{\lambda
-q^{\prime}}|\nabla\xi|^{q^{\prime}}\\
&  \leqq%
{\textstyle\int_{\Omega}}
u^{R}(.,s)\xi^{\lambda}+\varepsilon\int_{s}^{t}\int_{B_{2\eta}}u^{R}%
\xi^{\lambda}+\varepsilon^{1-R}\int_{s}^{t}\int_{B_{2\eta}}\xi^{\lambda
-Rq^{\prime}}|\nabla\xi|^{Rq^{\prime}}.
\end{align*}
Taking $\lambda=Rq^{\prime},$ and $\xi$ as above, we find%
\[
\int_{B_{2\eta}}u^{R}(.,t)\xi^{Rq^{\prime}}\leqq\int_{B_{2\eta}}u^{R}%
(.,s)\xi^{Rq^{\prime}}+\varepsilon\int_{s}^{t}\int_{B_{2\eta}}u^{R}%
\xi^{Rq^{\prime}}+\varepsilon^{1-R}C(N)C_{0}^{Rq^{\prime}}(N)\eta
^{N-Rq^{\prime}}t.
\]
Next we set $\varpi(t)=\sup_{\sigma\in\left[  s,t\right]  }\int_{B_{2\eta}%
}u^{R}(.,\sigma)\xi^{Rq^{\prime}}.$ Then
\[
\varpi(t)\leqq\int_{B_{2\eta}}u^{R}(.,s)\xi^{Rq^{\prime}}+\varepsilon
(t-s)\varpi(t)+\varepsilon^{1-R}C(N)C_{0}^{Rq^{\prime}}(N)\eta^{N-Rq^{\prime}%
}t.
\]
Taking $\varepsilon=1/2t,$ we get
\[
\frac{1}{2}\int_{B_{2\eta}}u^{R}(.,t)\xi^{Rq^{\prime}}\leqq\int_{B_{2\eta}%
}u^{R}(.,s)\xi^{rq^{\prime}}+C(N)C_{0}^{Rq^{\prime}}(N)\eta^{N-Rq^{\prime}%
}t^{R}.
\]
Then going to the limit as $s\rightarrow0,$%
\begin{equation}
\int_{B_{\eta}}u^{R}(x,t)\leqq C(N)C_{0}^{Rq^{\prime}}(N)\eta^{N-Rq^{\prime}%
}t^{R}+\int_{B_{2\eta}}u_{0}^{R}\xi^{Rq^{\prime}}, \label{zaf}%
\end{equation}
thus (\ref{locint}) follows.

If $u\in C(\overline{B_{2\rho}}\times\left[  0,T\right)  ),$ then (\ref{zaf})
holds for any $R\geqq1,$ implying
\[
\left\Vert u(.,t)\right\Vert _{L^{R}(B_{\eta})}\leqq C^{\frac{1}{R}}%
(N)C_{0}^{q^{\prime}}(N)\eta^{\frac{N}{R}-q^{\prime}}t+\left\Vert
u_{0}\right\Vert _{L^{R}(B_{2\eta})},
\]
and (\ref{locint}) follows as $R\rightarrow\infty$.
\end{proof}

\subsection{Regularizing effect of the heat operator}

We first give a first regularizing effect due to the Laplace operator in
$Q_{\Omega,T}$, for any domain $\Omega,$ for classical or weak solutions in
terms of the initial data.

\begin{theorem}
\label{local} Let $q>1.$ Let $u$ be any nonnegative weak subsolution of
equation (\ref{un}) in $Q_{\Omega,T}$, and let $B(x_{0},2\eta)\subset\Omega$
such that $u$ has a trace $u_{0}\in\mathcal{M}^{+}(B(x_{0},2\eta)).$ Then for
any $\tau<T,$ and any $t\in\left(  0,\tau\right]  ,$
\begin{equation}
\sup_{x\in B(x_{0},\eta/2)}u(x,t)\leqq Ct^{-\frac{N}{2}}(t+\int_{B(x_{0}%
,\eta)}du_{0}),\qquad C=C(N,q,\nu,\eta,\tau). \label{locma}%
\end{equation}
Moreover if $u_{0}\in L_{loc}^{R}(\Omega)$ $(R>1),$ and $u\in C(\left[
0,T\right)  ;L_{loc}^{R}(\Omega)),$ then
\begin{equation}
\sup_{x\in B(x_{0},\eta/2)}u(x,t)\leqq Ct^{-\frac{N}{2R}}(t+\left\Vert
u_{0}\right\Vert _{L^{R}(B(x_{0},\eta))}),\qquad C=C(N,q,\nu,R,\eta,\tau).
\label{locmo}%
\end{equation}

\end{theorem}

\begin{proof}
We still assume that $x_{0}=0\in\Omega.$ Let $\xi\in C_{c}^{1}(B_{2\eta})$ be
nonnegative, radial, with values in $\left[  0,1\right]  ,$ with $\xi=1$ on
$B_{\eta}$ and $|\nabla\xi|\leqq C_{0}(N)/\eta$. Since $u$ is $\nu
$-subcaloric, from (\ref{nusc}), for any $\rho\in(0,\eta)$ such that $\rho
^{2}\leqq t<\tau,$
\begin{equation}
\sup_{B_{\eta/2}}u(.,t)\leqq C(N,\nu)\rho^{-(N+2)}\int_{t-\rho^{2}/4}^{t}%
\int_{B_{\eta}}u, \label{moy}%
\end{equation}
hence from Lemma \ref{cor},
\[
\sup_{B_{\eta/2}}u(.,t)\leqq C(N,q,\nu)\rho^{-N}(\eta^{N-q^{\prime}}%
t+\int_{B_{2\eta}}du_{0}).
\]
Let $k_{0}\in\mathbb{N}$ such that $k_{0}\eta^{2}/2\geqq\tau.$ For any
$t\in\left(  0,\tau\right]  ,$ there exists $k\in\mathbb{N}$ with $k\leqq
k_{0}$ such that $t\in\left(  k\eta^{2}/2,(k+1)\eta^{2}/2\right]  .$ Taking
$\rho^{2}=t/(k+1),$ we find
\begin{align}
\text{sup}_{B_{\eta/2}}u(.,t)  &  \leqq C(N,q,\nu)(k_{0}+1)^{\frac{N}{2}%
}t^{-\frac{N}{2}}(\eta^{N-q^{\prime}}t+\int_{B_{2\eta}}du_{0})\nonumber\\
&  \leqq C(N,q,\nu)(\eta^{-N}\tau^{\frac{N}{2}}+1)t^{-\frac{N}{2}}%
(\eta^{N-q^{\prime}}t+\int_{B_{2\eta}}du_{0}). \label{ast}%
\end{align}
Thus we obtain (\ref{locma}). Next assume that $u\in C(\left[  0,T\right)
;L_{loc}^{R}(B_{2\eta})),$ with $R>1$. We still approximate $u$ by
$u_{\delta,n}=(u+\delta)\ast\varphi_{n}$, where $(\varphi_{n})$ is a sequence
of mollifiers, and $\delta>0.$ Since $u$ is $\nu$-subcaloric, then
$u_{\delta,n}^{R}$ is also $\nu$-subcaloric. Then for any $\rho\in(0,\eta)$
such that $\rho^{2}\leqq t<\tau,$ we have%
\[
\text{sup}_{B_{\eta/2}}u_{\delta,n}^{R}(.,t)\leqq C(N,\nu)\rho^{-(N+2)}%
\int_{t-\rho^{2}/4}^{t}\int_{B\rho/2}u_{\delta,n}^{R},
\]
hence as $\delta\rightarrow0$ and $n\rightarrow\infty,$ from Lemma
(\ref{cor}),
\begin{equation}
\text{sup}_{B_{\eta/2}}u^{R}(.,t)\leqq C(N,\nu)\rho^{-(N+2)}\int_{t-\rho
^{2}/4}^{t}\int_{B\rho/2}u^{R}\leqq C(N,q,\nu,R)(\eta^{-N}\tau^{\frac{N}{2}%
}+1)(\eta^{N-Rq^{\prime}}t^{R}+\int_{B_{2\eta}}u_{0}^{R}). \label{bst}%
\end{equation}
We deduce (\ref{locmo}) as above.
\end{proof}

\section{Global estimates in $\mathbb{R}^{N}$\label{sec4}}

We first show that the universal estimate of the gradient (\ref{versa})
implies the estimate (\ref{esta}) of the function:

\begin{theorem}
\label{top} Let $q>1.$ Let $u$ be a classical solution of equation (\ref{un})
in $Q_{\mathbb{R}^{N},T}.$ Assume that there exists a ball $B(x_{0},2\eta)$
such that $u$ has a trace $u_{0}\in\mathcal{M}^{+}((B(x_{0},2\eta)).$ If $u$
satisfies (\ref{versa}), then for any $t\in\left(  0,T\right)  ,$%
\begin{equation}
u(x,t)\leqq C(q)t^{-\frac{1}{q-1}}\left\vert x-x_{0}\right\vert ^{q^{\prime}%
}+C(t^{-\frac{1}{q-1}}+t+\int_{B(x_{0},\eta)}du_{0}),\qquad C=C(N,q,\eta),
\label{plec}%
\end{equation}
If \thinspace$u_{0}\in L_{loc}^{R}(\Omega),$ $R\geqq1$ and $u\in C(\left[
0,T\right)  ;L_{loc}^{R}(\Omega)),$ then
\begin{equation}
u(x,t)\leqq C(q)t^{-\frac{1}{q-1}}\left\vert x-x_{0}\right\vert ^{q^{\prime}%
}+Ct^{-\frac{N}{2R}}(t+\left\Vert u_{0}\right\Vert _{L^{R}(B(x_{0},\eta
))}),\qquad C=C(N,q,R,\nu,\eta). \label{plac}%
\end{equation}%
\begin{equation}
u(x,t)\leqq C(q)t^{-\frac{1}{q-1}}\left\vert x-x_{0}\right\vert ^{q^{\prime}%
}+C(t^{-\frac{1}{q-1}}+t+\left\Vert u_{0}\right\Vert _{L^{R}(B(x_{0},\eta
))}),\qquad C=C(N,q,R,\eta). \label{pluc}%
\end{equation}

\end{theorem}

\begin{proof}
Estimate (\ref{versa}) is equivalent to
\begin{equation}
\left\vert \nabla(u^{\frac{1}{q^{\prime}}})\right\vert (.,t)\leqq
\frac{(q-1)^{\frac{1}{q^{\prime}}}}{q}t^{-\frac{1}{q}},\qquad\qquad\text{in
}Q_{\mathbb{R}^{N},T}. \label{versi}%
\end{equation}
Then with constants $C(q)$ only depending of $q,$
\begin{equation}
u^{\frac{1}{q^{\prime}}}(x,t)\leqq u^{\frac{1}{q^{\prime}}}(x_{0}%
,t)+C(q)t^{-\frac{1}{q}}\left\vert x-x_{0}\right\vert , \label{masc}%
\end{equation}
then%
\begin{equation}
u(x,t)\leqq C(q)(u(x_{0},t)+t^{-\frac{1}{q-1}}\left\vert x-x_{0}\right\vert
^{q^{\prime}}), \label{misc}%
\end{equation}
and, from Theorem \ref{local},
\[
u(x_{0},t)\leqq C(N,q,R,\nu,\eta)t^{-\frac{N}{2R}}(t+\left\Vert u_{0}%
\right\Vert _{L^{R}(B(x_{0},\eta))}).
\]
Therefore (\ref{plac}) follows. Also, interverting $x$ and $x_{0},$ for any
$R\geqq1,$
\[
u^{R}(x_{0},t)\leqq C(q,R)(u^{R}(x,t)+t^{-\frac{R}{q-1}}\left\vert
x-x_{0}\right\vert ^{Rq^{\prime}}).
\]
Integrating on $B(x_{0},\eta/2),$ we get
\[
\eta^{N}u^{R}(x_{0},t)\leqq C(q,R)(\int_{B(x_{0},\eta/2)}u^{R}(.,t)+t^{-\frac
{R}{q-1}}\eta^{N-Rq^{\prime}});
\]
using Lemma \ref{cor}, we deduce
\[
u(x_{0},t)\leqq C(N,q,R,\eta)(t^{-\frac{1}{q-1}}+t+\int_{B(x_{0},\eta)}%
du_{0}),
\]
and if $u_{0}\in L_{loc}^{R}(\Omega),$
\[
u(x_{0},t)\leqq C(N,q,R,\eta)(t^{-\frac{1}{q-1}}+t+\left\Vert u_{0}\right\Vert
_{L^{R}(B(x_{0},\eta))}),
\]
and the conclusions follow from (\ref{misc}).\medskip
\end{proof}

\begin{remark}
In particular, the estimates (\ref{plec})-(\ref{pluc}) hold for solutions with
$u_{0}\in C_{b}(\mathbb{R}^{N}),$ and more generally for limits $a.e.$ of such
solutions, that we can call \textbf{reachable} solutions. Inegality
(\ref{masc}) was used in \cite[Theorem 3.3]{BeBALa} for obtaining local
estimates of classical of bounded solutions.in $Q_{\mathbb{R}^{N},T}.$
\end{remark}

In order to prove Theorem \ref{fund}, we first give an estimate of the type of
(\ref{esta}) on a time interval $\left(  0,\tau\right]  $, \textit{with
constants depending on} $\tau$ \textit{and} $\nu,$ which is \textit{not
obtained from any estimate of the gradient}. Our result is based on the
construction of suitable supersolutions in annulus of type $Q_{B_{3\rho
}\backslash\overline{B_{\rho}},\infty}$, $\rho>0.$ For the construction we
consider the function $t\in\left(  0,\infty\right)  \longmapsto\psi_{h}%
(t)\in(1,\infty),$ where $h>0$ is a parameter, solution of the problem%
\begin{equation}
(\psi_{h})_{t}+h(\psi_{h}^{q}-\psi_{h})=0\quad\text{in }\left(  0,\infty
\right)  ,\qquad\psi_{h}(0)=\infty,\quad\psi_{h}(\infty)=1, \label{jip}%
\end{equation}
given explicitely by $\psi_{h}(t)=(1-e^{-h(q-1)t})^{-\frac{1}{q-1}};$ hence
$\psi_{h}^{q}-\psi_{h}\geqq0,$ and for any $t>0,$
\begin{equation}
((q-1)ht)^{-\frac{1}{q-1}}\leqq\psi_{h}(t)\leqq2^{\frac{1}{q-1}}%
(1+((q-1)ht)^{-\frac{1}{q-1}}). \label{jop}%
\end{equation}
since, for $x>0,$ $x(1-x/2)\leqq1-e^{-x}\leqq x,$ hence $x/2\leqq1-e^{-x}\leqq
x,$ for $x\leqq1.$\medskip

\begin{proposition}
\label{plus}Let $q>1.$ Then there exists a nonnegative function $V$ defined in
$Q_{B_{3}\times(0,\infty)},$ such that $V$ is a supersolution of equation
(\ref{un}) on $Q_{B_{3}\backslash\overline{B_{1}},\infty}$,$,$ and $V$
converges to $\infty$ as $t\rightarrow0,$ uniformly on $B_{3}$ and converges
to $\infty$ as $x\rightarrow\partial B_{3},$ uniformly on $\left(
0,\tau\right)  $ for any $\tau<\infty.$ And $V$ has the form%
\begin{equation}
V(x,t)=e^{t}\Phi(\left\vert x\right\vert )\psi_{h}(t)\qquad\text{in }%
Q_{B_{3},\infty} \label{dev}%
\end{equation}
for some $h=h(N,q,\nu)>0,$ where $\psi_{h}$ is given by (\ref{jip}), and
$\Phi$ is a suitable radial function depending on $N,q,\nu,$ such that%
\begin{equation}
-\nu\Delta\Phi+\Phi+\left\vert \nabla\Phi\right\vert ^{q}\geqq0\qquad\text{in
}B_{3}. \label{sups}%
\end{equation}

\end{proposition}

\begin{proof}
We first construct $\Phi.$ Let $\sigma>0,$ such that $\sigma\geqq
a=(2-q)/(q-1).$ Let $\varphi_{1}$ be the first eigenfunction of the Laplacian
in $B_{3}$ such that $\varphi_{1}(0)=1,$ associated to the first eigenvalue
$\lambda_{1}$, hence $\varphi_{1}$ is radial ; let $m_{1}=\min_{\overline
{B_{1}}}\varphi_{1}>0$ and $M_{1}=\min_{\overline{B_{3}}\backslash B_{1}%
}\left\vert \nabla\varphi_{1}\right\vert .$ Let us take $\Phi=\Phi_{K}%
=\Phi_{0}+K,$ where $\Phi_{0}=\gamma\varphi_{1}^{-\sigma},$ $K>0$ and
$\gamma>0$ are parameters Then
\[
-\nu\Delta\Phi+\Phi+\left\vert \nabla\Phi\right\vert ^{q}=F(\Phi
_{0})+K,\text{\qquad with }%
\]%
\[
F(\Phi_{0})=\gamma\varphi_{1}^{-(\sigma+2)}(\gamma^{q-1}\sigma^{q}\varphi
_{1}^{(q-1)(a-\sigma)}\left\vert \varphi_{1}^{\prime}\right\vert ^{q}%
+(1-\nu\sigma\lambda_{1})\varphi_{1}^{2}-\nu\sigma(\sigma+1)\varphi
_{1}^{\prime2}).
\]
There holds $\lim_{r\rightarrow3}\left\vert \varphi_{1}^{\prime}\right\vert
=c_{1}>0$ from the H\"{o}pf Lemma. Taking $\sigma>a$ we fix $\gamma=1,$ and
then $\lim_{r\rightarrow3}F(\Phi_{0})=\infty.$ If $q<2$ we can also take
$\sigma=a,$ we get
\[
F(\Phi_{0})=\gamma\varphi_{1}^{-q\prime}(\gamma^{q-1}a^{q}\left\vert
\varphi_{1}^{\prime}\right\vert ^{q}+(1-\nu a\lambda_{1})\varphi_{1}%
^{2}-aq^{\prime}\varphi_{1}^{\prime2}),
\]
hence fixing $\gamma>\gamma(N,q,\nu)$ large enough, we still get
$\lim_{r\rightarrow3}F(\Phi_{0})=\infty.$ Thus $F$ has a minimum $\mu$ in
$B_{3}.$ Taking $K=K(N,q,\nu)>\left\vert \mu\right\vert $ we deduce that
$\Phi$ satisfies (\ref{sups}), and $\lim_{r\rightarrow3}\Phi=\infty.$

Observe that $\Phi^{\prime q}/\Phi=\gamma^{q}\sigma^{q}/(\gamma\varphi
_{1}^{q+\sigma(q-1)}+K\varphi_{1}^{q(\sigma+1)})$ is increasing, then
$m_{K}=m_{K}(N,q,\nu)=\min_{\left[  1,3\right]  }\left\vert \Phi^{\prime
}\right\vert ^{q}/\Phi=\left\vert \Phi^{\prime}(1)\right\vert ^{q}/\Phi(1)>0.$
We define $V$ by (\ref{dev}) and compute%
\begin{align*}
V_{t}-\nu\Delta V+\left\vert \nabla V\right\vert ^{q}  &  =e^{t}(\Phi\psi
_{h}+\Phi(\psi_{h})_{t}-\nu\Delta\Phi)+e^{qt}\left\vert \nabla\Phi\right\vert
^{q}\psi_{h}^{q}\\
&  \geqq e^{t}(\Phi\psi_{h}+\Phi\psi_{t}-\nu\Delta\Phi+\left\vert \nabla
\Phi\right\vert ^{q}\psi^{q})=e^{t}(\psi^{q}-\psi_{h})(\left\vert \nabla
\Phi\right\vert ^{q}-h\Phi).
\end{align*}
We take $h=h(N,q,\nu)<m_{K}.$ Then on $B_{3}\backslash B_{1}$ we have
$\left\vert \nabla\Phi\right\vert ^{q}-h\Phi>0,$ and $\psi^{q}\geqq\psi_{h},$
then $V$ is a supersolution on $B_{3}\backslash B_{1}.$ Moreover $V$ is radial
and increasing with respect to $\left\vert x\right\vert ,$ then
\begin{align}
\sup_{\overline{B_{2}}}V(x,t)  &  =\sup_{\overline{\partial B_{2}}%
}V(x,t)=e^{t}\Phi(2)\psi_{h}(t)\leqq2^{\frac{1}{q-1}}e^{t}\Phi
(2)(1+((q-1)ht)^{-\frac{1}{q-1}})\nonumber\\
&  \leqq C(N,q,\nu)e^{t}\Phi(2)(1+t^{-\frac{1}{q-1}}). \label{ahh}%
\end{align}

\end{proof}

\begin{theorem}
\label{nice}Let $u$ be a classical solution, (in particular any weak solution
if $q\leqq2)$ of equation (\ref{un}) in $Q_{\mathbb{R}^{N},T}$ . Assume that
there exists a ball $B(x_{0},2\eta)$ such that $u$ admits a trace $u_{0}%
\in\mathcal{M}^{+}(B(x_{0},2\eta)).$

(i) Then for any $\tau\in\left(  0,T\right)  ,$ and $t\leqq\tau,$%
\begin{equation}
u(x,t)\leqq C(t^{-\frac{1}{q-1}}\left\vert x-x_{0}\right\vert ^{q^{\prime}%
}+t^{-\frac{N}{2}}(t+\int_{B(x_{0},\eta)}du_{0})),\qquad C=C(N,q,\nu,\eta
,\tau), \label{bac}%
\end{equation}
(ii) Also if $u\in C(\left[  0,T\right)  ;L_{loc}^{R}(B(x_{0},2\eta))),$
\begin{equation}
u(x,t)\leqq C(t^{-\frac{1}{q-1}}\left\vert x-x_{0}\right\vert ^{q^{\prime}%
}+t^{-\frac{N}{2R}}(t+\left\Vert u_{0}\right\Vert _{L^{R}(B(x_{0},\eta
))})),\qquad C=C(N,q,\nu,R,\eta,\tau), \label{bec}%
\end{equation}
if $u\in C(\left[  0,T\right)  \times B(x_{0},2\eta)),$ then
\begin{equation}
u(x,t)\leqq C(t^{-\frac{1}{q-1}}\left\vert x-x_{0}\right\vert ^{q^{\prime}%
}+t+\sup_{B(x_{0},\eta)}u_{0}),\qquad C=C(N,q,\nu,\eta,\tau). \label{bic}%
\end{equation}

\end{theorem}

\begin{proof}
We use the function $V$ constructed above. We can assume $x_{0}=0$. For any
$\rho>0,$ we consider the function $V_{\rho}$ defined in $B_{3\rho}%
\times(0,\infty)$ by
\[
V_{\rho}(x,t)=\rho^{-a}V(\rho^{-1}x,\rho^{-2}t).
\]
It is a supersolution of the equation (\ref{un}) on $B_{3\rho}\backslash
\overline{B_{\rho}}\times(0,\infty),$ infinite on $\partial B_{3\rho}%
\times(0,\infty)$ and on $B_{3\rho}\times\left\{  0\right\}  ,$ and from
(\ref{ahh})
\begin{align}
\sup_{\overline{B_{2\rho}}}V_{\rho}(x,t)  &  =\sup_{\partial B_{2\rho}}%
V_{\rho}(x,t)\leqq C_{1}(N,q,\nu)\rho^{-a}e^{\frac{t}{\rho^{2}}}\Phi
(2)(1+\rho^{\frac{2}{q-1}}t^{-\frac{1}{q-1}})\nonumber\\
&  \leqq C_{2}(N,q,\nu)\rho^{q^{\prime}}e^{\frac{t}{\rho^{2}}}(\rho^{-\frac
{2}{q-1}}+t^{-\frac{1}{q-1}}). \label{alpa}%
\end{align}

(i) First suppose that $u\in C(\left[  0,T\right)  \times\mathbb{R}^{N})).$
Let $\tau\in\left(  0,T\right)  ,$ and $C(\tau)=$ $\sup_{Q_{B_{\rho},\tau}}u$.
Then $w=C(\tau)+V_{\rho}$ is a supersolution in $Q=(B_{3\rho}\backslash
\overline{B_{\rho}})\times\left(  0,\tau\right]  ,$ and from the comparison
principle we obtain $u\leqq C(\tau)+V_{\rho}$ in that set. Indeed let
$\epsilon>0$ small enough. Then there exists $\tau_{\epsilon}<\epsilon$ and
$r_{\epsilon}\in\left(  3\rho-\epsilon,3\rho\right)  $, such that
$w(.,s)\geqq\max_{\overline{B_{3\rho}}}u(.,\epsilon)$ for any $s\in\left(
0,\tau_{\epsilon}\right]  $, and $w(x,t)\geqq\max_{\overline{B_{3\rho}}%
\times\left[  0,\tau\right]  }u$ for any $t\in\left(  0,\tau\right]  $ and
$r_{\epsilon}\leqq\left\vert x\right\vert <3\rho.$ We compare $u(x,t+\epsilon
)$ with $w(x,t+s)$ on $\left[  0,\tau-\epsilon\right]  \times\overline
{B_{r_{\epsilon}}}\backslash\overline{B_{\rho}}.$ And for $\left\vert
x\right\vert =\rho,$ we have $u(x,t+\epsilon)\leqq C(\tau)\leqq w(x,t+s)$.
Then $u(.,t+\epsilon)\leqq w(.,t+s)$ in $\overline{B_{r_{\epsilon}}}%
\backslash\overline{B_{\rho}}\times\left(  0,\tau-\epsilon\right]  $. As
$s,\epsilon\rightarrow0,$ we deduce that $u\leqq w$ in $Q$.

Hence in $\overline{B_{2\rho}}\times(0,\tau),$ we find from (\ref{alpa})
\begin{equation}
u\leqq C(\tau)+\sup_{\overline{B_{2\rho}}}V_{\rho}(x,t)\leqq C(\tau)+C_{2}%
\rho^{q^{\prime}}e^{\frac{t}{\rho^{2}}}(\rho^{-\frac{2}{q-1}}+t^{-\frac
{1}{q-1}}). \label{sho}%
\end{equation}
Making $t$ tend to $\tau,$ this proves that
\[
\sup_{Q_{B_{2\rho},\tau}}u\leqq\sup_{Q_{B_{\rho},\tau}}u+C_{2}\rho^{q^{\prime
}}e^{\frac{\tau}{\rho^{2}}}(\rho^{-\frac{2}{q-1}}+\tau^{-\frac{1}{q-1}})
\]
By induction, we get%
\begin{align*}
\sup_{Q_{B_{2^{n+1}\rho}},\tau}u  &  \leqq\sup_{Q_{B_{2^{n}\rho}},\tau}%
u+C_{2}2^{nq^{\prime}}\rho^{q^{\prime}}e^{\frac{\tau}{4^{n}\rho^{2}}}%
((2^{n}\rho)^{-\frac{2}{q-1}}+\tau^{-\frac{1}{q-1}})\\
&  \leqq\sup_{Q_{B_{2^{n}\rho}},\tau}u+C_{2}2^{nq^{\prime}}\rho^{q^{\prime}%
}e^{\frac{\tau}{\rho^{2}}}(\rho^{-\frac{2}{q-1}}+\tau^{-\frac{1}{q-1}});
\end{align*}%
\begin{align*}
\sup_{Q_{B_{2^{n+1}\rho}},\tau}u  &  \leqq\sup_{Q_{B_{\rho}}}u+C_{2}%
(1+2^{q^{\prime}}+..+2^{nq^{\prime}})\rho^{q^{\prime}}e^{\frac{\tau}{\rho^{2}%
}}(\rho^{-\frac{2}{q-1}}+\tau^{-\frac{1}{q-1}})\\
&  \leqq\sup_{Q_{B_{\rho},\tau}}u+C_{2}2^{q^{\prime}}(2^{n}\rho)^{q^{\prime}%
}e^{\frac{\tau}{\rho^{2}}}(\rho^{-\frac{2}{q-1}}+\tau^{-\frac{1}{q-1}}).
\end{align*}
For any $x\in\mathbb{R}^{N}$ such that $\left\vert x\right\vert \geqq\rho,$
there exists $n\in\mathbb{N}^{\ast}$ such that $x\in B_{2^{n+1}\rho}%
\backslash\overline{B_{2^{n}\rho}},$ then
\begin{equation}
u(x,\tau)\leqq\sup_{Q_{B_{\rho}},\tau}u+C_{2}2^{q^{\prime}}\left\vert
x\right\vert ^{q^{\prime}}e^{\frac{\tau}{\rho^{2}}}(\rho^{-\frac{2}{q-1}}%
+\tau^{-\frac{1}{q-1}}) \label{moc}%
\end{equation}
thus
\begin{equation}
\sup_{Q_{\mathbb{R}^{N}},\tau}u\leqq\sup_{Q_{B_{\rho}},\tau}u+C_{2}%
2^{q^{\prime}}\left\vert x\right\vert ^{q^{\prime}}e^{\frac{\tau}{\rho^{2}}%
}(\rho^{-\frac{2}{q-1}}+\tau^{-\frac{1}{q-1}}). \label{mic}%
\end{equation}
\bigskip

(ii) Next we consider any classical solution $u$ in $Q_{\mathbb{R}^{N},T}$
with trace $u_{0}$ in $B(x_{0},2\eta)$. We still assume $x_{0}=0.$ Then for
$0<\epsilon\leqq t\leqq\tau,$ from (\ref{locinf}) in Lemma \ref{cor}, there
holds%
\[
\sup_{B_{\eta}/2}u(x,t)\leqq C(N,q)\eta^{-q^{\prime}}t+\sup_{B_{\eta}%
}u(x,\epsilon).
\]
Then from (\ref{mic}) with $\rho=\eta/2$, we deduce that for any $(x,t)\in
Q_{\mathbb{R}^{N},\epsilon,\tau},$
\[
u(x,t)\leqq C(N,q)\eta^{-q^{\prime}}t+\sup_{B_{\eta/2}}u(.,\epsilon
)+C(1+(t-\epsilon)^{-\frac{1}{q-1}})\left\vert x\right\vert ^{q^{\prime}},
\]
with $C=C(N,q,\nu,\eta,\tau).$ Next we take $\epsilon=t/2$. Then for any
$t\in\left(  0,\tau\right]  ,$ from (\ref{locma}) in Theorem \ref{local},
\[
u(x,t)\leqq C(N,q,\eta)t+Ct^{-1(q-1)}\left\vert x\right\vert ^{q^{\prime}%
}+Ct^{-\frac{N}{2}}(t+\int_{B_{\eta}}du_{0}).
\]
with $C=C(N,q,\nu,\eta,\tau)$ and we obtain (\ref{bac}). And (\ref{bec}),
(\ref{bic}) follow from (\ref{locmo}) and (\ref{locinf}).\medskip
\end{proof}

Next we show our main Theorem \ref{fund}. We use a \textit{local} Bernstein
technique, as in \cite{SoZh}. The idea is to compute the equation satisfied by
the function $v=u^{(q-1)/q}$ , introduced in \cite{BeLa99}, and the equation
satisfied by $w=\left\vert \nabla v\right\vert ^{2},$ to obtain estimates of
$w$ in a cylinder $Q_{B_{M},T},$ $M>0.$ The difficulty is that this equation
involves an elliptic operator $w\mapsto w_{t}-\Delta w+b.\nabla w,$ where $b$
depends on $v,$ and may be unbounded. However it can be controlled by the
estimates of $v$ obtained at Theorem \ref{nice}. Then as $M\rightarrow\infty,$
we can prove nonuniversal $L^{\infty}$ estimates of $w.$ Finally we obtain
universal estimates of $w$ by application of the maximum principle in
$Q_{\mathbb{R}^{N},T},$ valid because $w$ is bounded$.\medskip$ First we give
a slight improvement of a comparison principle shown in \cite[Proposition
2.2]{SoZh}.

\begin{lemma}
\label{prin} Let $\Omega$ be any domain of $\mathbb{R}^{N},$ and $\tau
,\kappa\in(0,\infty)$, $A,B\in\mathbb{R}.$ Let $U\in C(\left[  0,\tau\right)
;L_{loc}^{2}(\overline{\Omega}))$ such that $U_{t},\nabla u,D^{2}u\in
L_{loc}^{2}(\overline{\Omega}\times\left(  0,\tau\right)  ),$ $\mathrm{ess}%
\sup_{Q_{\Omega,\tau}}U<\infty,$ $U\leqq B$ on the parabolic boundary of
$Q_{\Omega,\tau},$ and
\[
U_{t}-\Delta U\leqq\kappa(1+\left\vert x\right\vert )\left\vert \nabla
U\right\vert +f\qquad\text{in }Q_{\Omega,\tau}%
\]
where $f=f(x,t)$ such that $f(.,t)\in L_{loc}^{2}(\overline{\Omega})$ for
$a.e.$ $t\in(0,\tau)$ and $f\leqq0$ on $\left\{  (x,t)\in Q_{\Omega,\tau
}:U(x,t)\geqq A\right\}  .$ Then ess$\sup_{Q_{\Omega,\tau}}U\leqq\max(A,B).$
\end{lemma}

\begin{proof}
We set $\varphi(x,t)=\Lambda t+\ln(1+\left\vert x\right\vert ^{2}),$
$\Lambda>0.$ Then $\nabla\varphi=2x/(1+\left\vert x\right\vert ^{2}),$
$0\leqq\Delta\varphi\leqq2N/(1+\left\vert x\right\vert ^{2})\leqq2N.$ Let
$\varepsilon>0$ and $Y=U-\max(A,B)-\varepsilon\varphi.$ Taking $\Lambda
=2\sqrt{2}\kappa+2N,$ we obtain
\[
Y_{t}-\Delta Y-f-\kappa(1+\left\vert x\right\vert )\left\vert \nabla
Y\right\vert \leqq\varepsilon(K(1+\left\vert x\right\vert )\left\vert
\nabla\varphi\right\vert -\varphi_{t}+\Delta\varphi)\leqq\varepsilon(2\sqrt
{2}\kappa+2N-\Lambda)=0.
\]
Since ess$\sup_{Q_{\Omega,\tau}}U<\infty,$ for $R$ large enough, and any
$t\in(0,\tau),$ we have $Y(.,t)\leqq0$ $a.e.$ in $\Omega\cap\left\{
\left\vert x\right\vert >R\right\}  .$ And $Y^{+}\in C(\left[  0,\tau\right)
;L^{2}(\Omega))\cap W^{1,2}((0,\tau);L^{2}(\Omega)),$ $Y^{+}(0)=0$ and
$Y^{+}(.,t)\in W^{1,2}(\Omega\cap B_{R})$ for $a.e.$ $t\in(0,\tau),$ and
$fY^{+}(.,t)\leqq0.$ Then
\[
\frac{1}{2}\frac{d}{dt}(\int_{\Omega}Y^{+2}(.,t)\leqq-\int_{\Omega}\left\vert
\nabla Y^{+}(.,t)\right\vert ^{2}+\kappa(1+R)\int_{\Omega}\left\vert \nabla
Y(.,t)\right\vert Y^{+}(.,t)\leqq\frac{\kappa^{2}(1+R)^{2}}{4}\int_{\Omega
}Y^{+2}(.,t),
\]
hence by integration $Y\leqq0$ $a.e.$ in $Q_{\Omega,\tau}.$ We conclude as
$\varepsilon\rightarrow0.$ \medskip
\end{proof}

\begin{proof}
[Proof of Theorem \ref{fund}]We can assume $x_{0}=0.$ By setting
$u(x,t)=\nu^{q^{\prime}/2}U(x/\sqrt{\nu},t),$ for proving (\ref{versi}) we can
suppose that $u$ is a classical solution of (\ref{un}) with $\nu=1.$ We set
\[
\delta+u=v^{\frac{q}{q-1}},\qquad\delta\in\left(  0,1\right)  .
\]

\textbf{(i) Local problem relative to} $\left\vert \nabla v\right\vert ^{2}.$
Here $u$ is any classical solution $u$ of equation (\ref{un}) in a cylinder
$Q_{B_{M},T}$ with $M>0.$ Then $v$ satisfies the equation
\begin{equation}
v_{t}-\Delta v=\frac{1}{q-1}\frac{\left\vert \nabla v\right\vert ^{2}}%
{v}-cv\left\vert \nabla v\right\vert ^{q},\qquad c=(q^{\prime})^{q-1}%
.\label{hav}%
\end{equation}
Setting $w=\left\vert \nabla v\right\vert ^{2},$ we define
\[
\mathcal{L}w=w_{t}-\Delta w+b.\nabla w,\qquad b=(qcvw^{\frac{q-2}{2}}-\frac
{2}{q-1}\frac{1}{v})\nabla v.
\]
Differentiating (\ref{hav}) and using the identity $\Delta w=2\nabla(\Delta
w).\nabla w+2\left\vert D^{2}v\right\vert ^{2},$ we obtain the equation
\begin{equation}
\mathcal{L}w+2cw^{\frac{q+2}{2}}+2\left\vert D^{2}v\right\vert ^{2}+\frac
{2}{q-1}\frac{w^{2}}{v^{2}}=0.\label{eqw}%
\end{equation}
As in \cite{SoZh}, for $s\in\left(  0,1\right)  ,$ we consider a test function
$\zeta\in C^{2}(\overline{B}_{3M/4})$ with values in $\left[  0,1\right]  ,$
$\zeta=0$ for $\left\vert x\right\vert \geq3M/4$ and $\left\vert \nabla
\zeta\right\vert \leqq C(N,s)\zeta^{s}/M$ and $\left\vert \Delta
\zeta\right\vert +\left\vert \nabla\zeta\right\vert ^{2}/\zeta\leqq
C(N,s)\zeta^{s}/M^{2}$ in $B_{3M/4}.$ We set $z=w\zeta.$ We have
\[
\mathcal{L}z=\zeta\mathcal{L}w+w\mathcal{L}\zeta-2\nabla w.\nabla\zeta
\leqq\zeta\mathcal{L}w+w\mathcal{L}\zeta+\left\vert D^{2}v\right\vert
^{2}\zeta+4w\frac{\left\vert \nabla\zeta\right\vert ^{2}}{\zeta}.
\]
It follows that in $Q_{B_{M},T},$
\[
\mathcal{L}z+2cw^{\frac{q+2}{2}}\zeta+\frac{2}{q-1}\frac{w^{2}}{v^{2}}%
\zeta\leqq\frac{C\zeta^{s}w}{M^{2}}+\frac{C\zeta^{s}w^{\frac{3}{2}}}%
{M}\left\vert cqvw^{\frac{q-2}{2}}-\frac{2}{q-1}\frac{1}{v}\right\vert \leqq
C\zeta^{s}(\frac{w}{M^{2}}+\frac{vw^{\frac{q+1}{2}}}{M}+\frac{w^{\frac{3}{2}}%
}{Mv}),
\]
with constants $C=C(N,q,s).$ Since $\zeta\leqq1,$ from the Young inequality,
taking $s\geqq\max(q+1),3)/(q+2),$ for any $\varepsilon>0,$
\[
\frac{C}{M}\zeta^{s}vw^{\frac{q+1}{2}}=\frac{C}{M}\zeta^{\frac{q+1}{q+2}}%
\zeta^{s-\frac{q+1}{q+2}}vw^{\frac{q+1}{2}}\leqq\varepsilon\zeta w^{\frac
{q+2}{2}}+C(N,q,\varepsilon)\frac{v^{q+2}}{M^{q+2}},
\]
and
\[
\frac{C}{M^{2}}\zeta^{s}w\leqq\varepsilon\zeta w^{\frac{q+2}{2}}%
+C(N,q,\varepsilon)\frac{1}{M^{\frac{2(q+2)}{q}}},
\]%
\[
\frac{C}{M}\zeta^{s}\frac{w^{\frac{3}{2}}}{v}\leqq\frac{1}{\delta M}\zeta
^{s}w^{\frac{3}{2}}=\frac{1}{\delta M}\zeta^{s-\frac{3}{q+2}}\zeta^{\frac
{3}{q+2}}w^{\frac{3}{2}}\leqq\varepsilon\zeta w^{\frac{q+2}{2}}%
+C(N,q,\varepsilon)\frac{1}{(\delta M)^{\frac{q+2}{q-1}}}.
\]
Then with a new $C=C(N,q,\delta)$
\begin{equation}
\mathcal{L}z+cz^{\frac{q+2}{2}}\leqq C(\frac{v^{q+2}}{M^{q+2}}+\frac
{1}{M^{\frac{2(q+2)}{q}}}+\frac{1}{M^{\frac{q+2}{q-1}}}).\label{shic}%
\end{equation}
\textbf{(ii) Nonuniversal estimates of }$w.$ Here we assume that $u$ is a
classical solution of (\ref{un}) in whole $Q_{\mathbb{R}^{N},T},$ such that
$u\in C(\mathbb{R}^{N}\times\left[  0,T\right)  )$. From Theorem \ref{nice},
for any $\tau\in(0,T),$ there holds in $Q_{\mathbb{R}^{N},\tau}$
\begin{equation}
v(x,t)=(\delta+u(x,t))^{\frac{q-1}{q}}\leqq C(t^{-\frac{1}{q}}\left\vert
x\right\vert +(t+\sup_{B_{2\eta}}u_{0})^{\frac{q-1}{q}}),\qquad C=C(N,q,\eta
,\tau).\label{jom}%
\end{equation}
hence for $M\geqq M(q,\sup_{B_{2\eta}}u_{0},\tau)\geqq1,$ we deduce
\[
v(x,t)\leqq2Ct^{-\frac{1}{q}}M,\qquad\text{in }Q_{B_{M},\tau}.
\]
Then with a new constant $C=C(N,q,\eta,\tau,\delta),$ there holds in
$Q_{B_{3M/4},\tau}$
\begin{equation}
\mathcal{L}z+cz^{\frac{q+2}{2}}\leqq Ct^{-\frac{q+2}{q}}.\label{glu}%
\end{equation}
Next we consider $\Psi(t)=Kt^{-2/q}.$ It satisfies
\[
\Psi_{t}+c\Psi^{\frac{q+2}{2}}=(cK^{\frac{q+2}{2}}-2q^{-1}K)t^{-\frac{q+2}{q}%
}\geqq Ct^{-\frac{q+2}{q}}%
\]
if $K\geqq\overline{K}=\overline{K}(N,q,\eta,\tau,\delta).$ Fixing
$\epsilon\in\left(  0,T\right)  $ such that $\tau+\epsilon<T,$ there exists
$\tau_{\epsilon}\in\left(  0,\epsilon\right)  $ such that $\Psi(\theta
)\geqq\sup_{B_{M}}z(.,\epsilon)$ for any $\theta\in(0,\tau_{\epsilon}).$ We
have
\begin{align*}
&  z_{t}(.,t+\epsilon)-\Delta z(.,t+\epsilon)+b(.,t+\epsilon).\nabla
(z,t+\epsilon)+cz^{\frac{q+2}{2}}(t+\epsilon)\\
&  \leqq C(t+\epsilon)^{-\frac{q+2}{q}}\leqq C(t+\theta)^{-\frac{q+2}{q}}%
\leqq\Psi_{t}(t+\theta)+c\Psi^{\frac{q+2}{2}}(t+\theta).
\end{align*}
Therefore, setting $\tilde{z}(.,t)=z(.,t+\epsilon)-\Psi(t+\theta),$ there
holds
\[
\tilde{z}(.,t)-\Delta\tilde{z}(.,t)+b(.,t+\epsilon).\nabla\tilde{z}(.,t)\leqq0
\]
on the set $\mathcal{V}=\left\{  (x,t)\in Q_{B_{3M/4},\tau+\epsilon}:\tilde
{z}(x,t)\geqq0\right\}  ;$ otherwise $\tilde{z}(.,t)\leqq0$ for sufficiently
small $t>0,$ and $\tilde{z}\leqq0$ on $\partial B_{3M/4}\times\left[
0,\tau\right]  .$ Then from Lemma \ref{prin}, we get $z(.,t+\epsilon)\leqq
\Psi(t+\theta)$ in $Q_{B_{3M/4},\tau},$ since $\left\vert b\right\vert
\leqq(qcvw^{\frac{q-1}{2}}+\frac{2}{q-1}\frac{1}{\delta}w^{1/2}),$ hence
bounded on $Q_{B_{3M/4},\tau+\epsilon}.$ Going to the limit as $\theta
,\epsilon\rightarrow0,$ we deduce that $z(.,t)\leqq\overline{K}t^{-\frac{2}%
{q}}$ in $Q_{B_{3M/4},\tau},$ thus $w(.,t)\leqq\overline{K}t^{-\frac{2}{q}}$
in $Q_{B_{M/2},\tau}.$ Next we go to the limit as $M\rightarrow\infty$ and
deduce that $w(.,t)\leqq\overline{K}t^{-\frac{2}{q}}$ in $Q_{\mathbb{R}%
^{N},\tau}$ , namely
\[
(q^{\prime})^{q}\left\vert \nabla v(.,t)\right\vert ^{q}=\frac{\left\vert
\nabla u\right\vert ^{q}}{\delta+u}(.,t)\leqq Ct^{-1},\qquad C=C(N,q,\eta
,\delta,\tau).
\]
In turn for any $\epsilon$ as above, \textit{there holds} $w\in L^{\infty
}(Q_{\mathbb{R}^{N},\epsilon,T})$, that means $\left\vert \nabla v\right\vert
\in L^{\infty}(Q_{\mathbb{R}^{N},\epsilon,\tau}).$\medskip

\textbf{(iii) Universal estimate (\ref{versi}) for }$u\in C(\mathbb{R}%
^{N}\times\left[  0,T\right)  ):$ we prove the universal estimate
(\ref{versi}). Taking again $\Psi(t)=Kt^{-2/q},$ with now $K=K(N,q)=q^{-2}%
(q-1)^{2/q^{\prime}},$ we have
\[
\Psi_{t}+2c\Psi^{\frac{q+2}{2}}\geqq(2cK^{\frac{q+2}{2}}-2q^{-1}%
K)t^{-\frac{q+2}{q}}\geqq0.
\]
And $\mathcal{L}w+2cw^{\frac{q+2}{2}}\leqq0$ from (\ref{eqw}). Moreover there
exists $\tau_{\epsilon}\in\left(  0,\tau\right)  $ such that $\Psi
(\theta)\geqq\sup_{\mathbb{R}^{N}}w(.,\epsilon)$ for any $\theta\in
(0,\tau_{\epsilon}).$ Setting $y(.,t)=$ $w(.,t+\epsilon)-\Psi(.,t+\theta),$
hence on the set $\mathcal{U}=\left\{  (x,t)\in Q_{\mathbb{R}^{N},\tau
}:y(x,t)\geqq0\right\}  ,$ there holds in the same way
\[
y(.,t)-\Delta y(.,t)+b(.,t+\epsilon).\nabla y(.,t)\leqq0.
\]
Here we only have from (\ref{jom})
\[
\left\vert b\right\vert \leqq(qcvw^{\frac{q-1}{2}}+\frac{2}{q-1}\frac
{1}{\delta}w^{1/2})\leqq\kappa_{\epsilon}(1+\left\vert x\right\vert )
\]
on $Q_{\mathbb{R}^{N},\epsilon,\tau}$, for some $\kappa_{\epsilon}%
=\kappa_{\epsilon}(N,q,\eta,\sup_{B_{2\eta}}u_{0},\tau,\epsilon).$ It is
sufficient to apply Lemma \ref{prin}. We deduce that $w(.,t+\epsilon)\leqq
\Psi(t+\theta)$ on $(0,\tau).$ As $\theta,\epsilon\rightarrow0$ we obtain that
$w(.,t)\leqq\Psi(t)=q^{-2}(q-1)^{2/q^{\prime}}t^{-2/q},$ which shows now that
in $(0,T)$
\[
\left\vert \nabla v(.,t)\right\vert ^{q}=(q^{\prime})^{-q}\frac{\left\vert
\nabla u\right\vert ^{q}}{\delta+u}(.,t)\leqq q^{-q}(q-1)^{(q-1)}t^{-1}.
\]
As $\delta\rightarrow0,$ we obtain (\ref{versi}). \medskip

\textbf{(iv) General universal estimate.} Here we relax the assumption $u\in
C(\mathbb{R}^{N}\times\left[  0,T\right)  ):$ For any $\epsilon\in\left(
0,T\right)  $ such that $\tau+\epsilon<T,$ we have $u\in C(\mathbb{R}%
^{N}\times\left[  \epsilon,\tau+\epsilon\right)  ),$ then from above,
\[
\left\vert \nabla v(.,t+\epsilon)\right\vert ^{q}\leqq\frac{1}{q-1}\frac{1}%
{t},
\]
and we obtain (\ref{versi}) as $\epsilon\rightarrow0,$ on $(0,\tau)$ for any
$\tau<T,$ hence on $(0,T).$\medskip
\end{proof}

\begin{proof}
[Proof of Theorem \ref{growth}]It is a direct consequence of Theorems
\ref{fund} and \ref{top}.
\end{proof}

\section{Existence and nonuniqueness results \label{sec5}}

First mention some known uniqueness and comparison results, for the Cauchy
problem, see \cite[Theorems 2.1,4.1,4.2 and Remark 2.1 ]{BASoWe},\cite[Theorem
2.3, 4.2, 4.25, Proposition 4.26 ]{BiDao2}, and for the Dirichlet problem, see
\cite[Theorems 3.1, 4.2]{Al}, \cite{BeDa}, \cite[Proposition 5.17]{BiDao2},
\cite{Po}.

\begin{theorem}
\label{souc} Let $\Omega=\mathbb{R}^{N}$ (resp. $\Omega$ bounded). (i) Let
$1<q<q_{\ast},$ and $u_{0}\in\mathcal{M}_{b}(\mathbb{R}^{N})($resp. $u_{0}%
\in\mathcal{M}_{b}(\Omega)$). Then there exists a unique weak solution $u$ of
(\ref{un}) with trace $u_{0}$ (resp. a weak solution of $(D_{\Omega,T}),$ such
that $\lim_{t\rightarrow0}u(.t)=u_{0}$ weakly in $\mathcal{M}_{b}(\Omega))$).
If $v_{0}\in\mathcal{M}_{b}(\Omega)$ and $u_{0}\leqq v_{0},$ and $v$ is the
solution associated to $v_{0},$ then $u\leqq v.$\medskip

(ii) Let $u_{0}\in L^{R}\left(  \Omega\right)  ,$ $1\leqq R\leqq\infty.$ If
$1<q<(N+2R)/(N+R),$ or if $q=2,$ $R<\infty,$ there exists a unique weak
solution $u$ of (\ref{un}) (resp. $(D_{\Omega,T})$) such that $u\in C(\left[
0,T\right)  ;L^{R}\left(  \Omega\right)  $ and $u(0)=u_{0}$. If $v_{0}\in
L^{R}\left(  \mathbb{R}^{N}\right)  $ and $u_{0}\leqq v_{0},$ then $u\leqq v.$
If $u_{0}$ is nonnegative, then for any $1<q\leqq2,$ there still exists at
least a weak nonnegative solution $u$ satisfying the same conditions.\medskip
\end{theorem}

Next we prove Theorem \ref{exim}. Our proof of (ii) (iii) is based on
approximations by nonincreasing sequences. Another proof can be obtained when
$u_{0}\in L_{loc}^{1}\left(  \mathbb{R}^{N}\right)  $ and $q\leqq2,$ by
techniques of equiintegrability, see \cite{LePe} for a connected problem.
\medskip

\begin{proof}
[Proof of Theorem \ref{exim}] Assume $\Omega=\mathbb{R}^{N}$ (resp. $\Omega$ bounded).

(i) Case $1<q<q_{\ast}$, $u_{0}\in\mathcal{M}^{+}\left(  \mathbb{R}%
^{N}\right)  $ (resp. $\mathcal{M}^{+}\left(  \Omega\right)  $):  Let
$u_{0,n}=u_{0}\llcorner B_{n}$ (resp. $u_{0,n}=u_{0}\llcorner\overline
{\Omega_{1/n}^{\prime}},$ where $\Omega_{n}=\left\{  x\in\Omega:d(x,\partial
\Omega)>1/n\right\}  $, for $n$ large enough). From Theorem \ref{souc}, there
exists a unique weak solution $u_{n}$ of (\ref{un}) (resp. of $(D_{\Omega,T}%
)$) with trace $u_{0,n},$ and $(u_{n})$ is nondecreasing; and $u_{n}\in
C^{2,1}(Q_{\mathbb{R}^{N},T})$ since $q\leqq2.$ From (\ref{ruc}), (\ref{zif}),
for any $\xi\in C_{c}^{1+}(\Omega),$
\begin{equation}
\int_{\Omega}u_{n}(.,t)\xi^{q^{\prime}}+\frac{1}{2}\int_{0}^{t}\int_{\Omega
}|\nabla u_{n}|^{q}\xi^{q^{\prime}}\leqq Ct\int_{\Omega}|\nabla\xi
|^{q^{\prime}}+\int_{\Omega}\xi^{q^{\prime}}du_{0}.\label{pic}%
\end{equation}
Hence $\left(  u_{n}\right)  $\textbf{ }is bounded in $L_{loc}^{\infty}\left(
\left[  0,T\right)  ;L_{loc}^{1}(\Omega)\right)  ,$ and $\left(  \left\vert
\nabla u_{n}\right\vert ^{q}\right)  $ is bounded in $L_{loc}^{1}\left(
\left[  0,T\right)  ;L_{loc}^{1}(\Omega)\right)  .$ In turn $\left(
u_{n}\right)  $\textbf{ }is bounded in $L_{loc}^{\infty}\left(  (0,T);L_{loc}%
^{\infty}(\Omega)\right)  ,$ from Theorem \ref{local}. From Theorem \ref{gul},
up to a subsequence, $(u_{n})$ converges in $C_{loc}^{2,1}(Q_{\mathbb{R}%
^{N},T})$ (resp. $C_{loc}^{2,1}(Q_{\Omega,T})\cap C^{1,0}\left(
\overline{\Omega}\times\left(  0,T\right)  \right)  $) to a weak solution $u$
of (\ref{un}) in $Q_{\mathbb{R}^{N},T}$ (resp. of $(D_{\Omega,T})$). Also from
\cite[Lemma 3.3]{BaPi}, for any $k\in\left[  1,q^{\ast}\right)  $ and any
$0<s<\tau<T,$
\[
\left\Vert u_{n}\right\Vert _{L^{k}((s,\tau);W^{1,k}(\omega))}\leqq
C(k,\omega)(\left\Vert u_{n}(s,.)\right\Vert _{L^{1}(\omega)}+\left\Vert
|\nabla u_{n}|^{q}+|\nabla u_{n}|+u_{n}\right\Vert _{L^{1}(Q_{\omega,s,\tau}%
)}),\qquad\forall\omega\subset\subset\Omega
\]
\[
\text{(resp. }\left\Vert u_{n}\right\Vert _{L^{k}((s,\tau);W_{0}^{1,k}%
(\Omega))}\leqq C(k,\Omega)(\left\Vert u_{n}(.,s)\right\Vert _{L^{1}(\Omega
)}+\left\Vert \left\vert \nabla u_{n}\right\vert ^{q}\right\Vert
_{L^{1}(Q_{\Omega,s,\tau})}).\text{)}%
\]
hence $\left(  u_{n}\right)  $ is bounded in $L_{loc}^{k}(\left[  0,T\right)
;W_{loc}^{1,k}(\mathbb{R}^{N}))$ (resp. $L_{loc}^{k}(\left[  0,T\right)
;W_{0}^{1,k}(\Omega))$). Since $q<q_{\ast},$ $(\left\vert \nabla
u_{n}\right\vert ^{q})$ is equiintegrable in $Q_{B_{M},\tau}$ for any $M>0$
(resp. in $Q_{\Omega,\tau})$ and $\tau\in\left(  0,T\right)  ,$ then $\left(
\left\vert \nabla u\right\vert ^{q}\right)  \in L_{loc}^{1}\left(  \left[
0,T\right)  ;L_{loc}^{1}(\Omega)\right)  $. From (\ref{bou}),
\begin{equation}
\int_{\Omega}u_{n}(t,.)\xi+\int_{0}^{t}\int_{\Omega}|\nabla u_{n}|^{q}%
\xi=-\int_{0}^{t}\int_{\Omega}\nabla u_{n}.\nabla\xi+\int_{\Omega}\xi
du_{0}.\label{zin}%
\end{equation}
As $n\rightarrow\infty$ we obtain
\[
\int_{\Omega}u(t,.)\xi+\int_{0}^{t}\int_{\Omega}|\nabla u|^{q}\xi=-\int%
_{0}^{t}\int_{\Omega}\nabla u.\nabla\xi+\int_{\Omega}\xi du_{0}.
\]
Thus $\lim_{t\rightarrow0}\int_{\Omega}u(.,t)\xi=\int_{\Omega}\xi du_{0},$ for
any $\xi\in C_{c}^{1+}(\Omega),$ hence for any $\xi\in C_{c}^{+}(\Omega);$
hence $u$ admits the trace $u_{0}.$

(ii) Case $q_{\ast}\leqq q\leqq2.$ Let us  set $u_{0,n}=\min(u_{0}%
,n)\chi_{B_{n}}$ (resp. $u_{0,n}=\min(u_{0},n)\chi_{\overline{\Omega
_{1/n}^{\prime}}}$ for $n$ large enough). Then $u_{0,n}\in L^{R}(\Omega)$ for
any $R\geqq1$. From Theorem \ref{souc}, the problem admits a solution $u_{n}$
, and it is unique in $C(\left[  0,T\right)  ;L^{R}\left(  \Omega\right)  )$
for any $R>(2-q)/N(q-1)$ and then $\left(  u_{n}\right)  $ is nondecreasing.
As above, $\left(  u_{n}\right)  $\textbf{ }is bounded in $L_{loc}^{\infty
}\left(  \left[  0,T\right)  ;L_{loc}^{1}(\Omega)\right)  ,$ $\left(
\left\vert \nabla u_{n}\right\vert ^{q}\right)  $ is bounded in $L_{loc}%
^{1}\left(  \left[  0,T\right)  ;L_{loc}^{1}(\Omega)\right)  ,$  $\left(
u_{n}\right)  $\textbf{ }is bounded in $L_{loc}^{\infty}\left(  (0,T);L_{loc}%
^{\infty}(\Omega)\right)  $ from Theorem \ref{local}. From Theorem \ref{gul},
$\left(  u_{n}\right)  $ converges in $C_{loc}^{2,1}(Q_{\Omega,T})$ to a weak
solution $u$ of (\ref{un}) in $Q_{\Omega,T},$ such that $u\in L_{loc}^{\infty
}\left(  \left[  0,T\right)  ;L_{loc}^{1}(\Omega)\right)  $ and $\left\vert
\nabla u\right\vert ^{q}\in$ $L_{loc}^{1}\left(  \left[  0,T\right)
;L_{loc}^{1}(\Omega)\right)  $.

Then from Remark \ref{trac}, $u$ admits a trace $\mu_{0}\in\mathcal{M}%
^{+}(\Omega)$ as $t\rightarrow0.$ Applying (\ref{zin}) to $u_{n},$ since
$u_{n}\leqq u,$ we get
\[
\lim_{t\rightarrow0}\int_{\Omega}u(.,t)\xi=\int_{\Omega}\xi d\mu_{0}\geqq
\lim_{t\rightarrow0}\int_{\Omega}u_{n}(.,t)\xi=\int_{\Omega}\xi du_{0},
\]
for any $\xi\in C_{c}^{+}(\Omega);$ thus $u_{0}\leqq\mu_{0}.$ Moreover
\[
\int_{\Omega}u_{n}(t,.)\xi+\int_{0}^{t}\int_{\Omega}|\nabla u_{n}|^{q}\xi
=\int_{0}^{t}\int_{\Omega}u_{n}\Delta\xi dx+\int_{\Omega}\xi du_{0}.
\]
And $\left(  u_{n}\right)  $ \textbf{ }is bounded in $L^{k}(Q_{\omega,\tau})$
for any $k\in\left(  1,q_{\ast}\right)  ;$ then for any domain $\omega
\subset\subset\Omega$, $(u_{n})$ converges strongly in $L^{1}(Q_{\omega,\tau
})$ ; then from the convergence $a.e.$ of the gradients, and the Fatou Lemma,%
\[
\int_{\mathbb{R}^{N}}u(t,.)\xi+\int_{0}^{t}\int_{\mathbb{R}^{N}}|\nabla
u|^{q}\xi\leqq\int_{0}^{t}\int_{\mathbb{R}^{N}}u\Delta\xi dx+\int%
_{\mathbb{R}^{N}}\xi du_{0}.
\]
But from Remark \ref{trac},
\[
\int_{\mathbb{R}^{N}}u(t,.)\xi+\int_{0}^{t}\int_{\mathbb{R}^{N}}|\nabla
u|^{q}\xi=\int_{0}^{t}\int_{\mathbb{R}^{N}}u\Delta\xi dx+\int_{\mathbb{R}^{N}%
}\xi d\mu_{0},
\]
then $\mu_{0}\leqq u_{0},$ hence $\mu_{0}=u_{0}.$ Finally we prove the
continuity: Let $\xi\in\mathcal{D}^{+}(\Omega)$ and $\omega\subset
\subset\Omega$ containing the support of $\xi.$ Then $z=u\xi$ is solution of
the Dirichlet problem
\[
\left\{
\begin{array}
[c]{l}%
z_{t}-\Delta z=g,\quad\text{in}\hspace{0.05in}Q_{\omega,T},\\
z=0,\quad\text{on}\hspace{0.05in}\partial\omega\times(0,T),\\
\lim_{t\rightarrow0}z(.,t)=\xi u_{0},\text{ weakly in }\mathcal{M}_{b}%
(\omega),
\end{array}
\right.
\]
with $g=-\left\vert \nabla u\right\vert ^{q}\xi+v(-\Delta\psi)-2\nabla
v.\nabla\psi\in L^{1}(Q_{\omega,T}).$ The solution is unique, see
\cite[Proposition 2.2]{BeDa}. Since $u_{0}\in L_{loc}^{1}\left(
\Omega\right)  ,$ there also exists a unique solution such that $z\in
C(\left[  0,T\right)  ,L^{1}(\omega))$ from \cite[Lemma 3.3]{BaPi}, hence
$u\in C(\left[  0,T\right)  ,L_{loc}^{1}(\Omega)).\medskip$

(iii) Case $q>2.$ We get the existence as above, by taking for $\left(
u_{0,n}\right)  $ a nondecreasing sequence in $C_{b}\left(  \mathbb{R}%
^{N}\right)  $ (resp. in $C_{0}\left(  \Omega\right)  ),$ converging to
$u_{0},$ and using Remark \ref{gil} for classical solutions.\medskip
\end{proof}

Next we show the nonuniqueness of the weak solutions when $q>2:$ here the
coefficient $a$ defined at (\ref{vala}) is negative, and $\left\vert
a\right\vert =(q-2)/(q-1)<1.$\medskip

\begin{proof}
[Proof of Theorem \ref{twosol}]Since $q>2$ and $N\geq2,$ the function
$\tilde{U}$ is a solution in $\mathcal{D}^{\prime}\left(  \mathbb{R}%
^{N}\right)  $ of the stationary equation%
\[
-\Delta u+|\nabla u|^{q}=0
\]
Indeed $\tilde{U}\in W_{loc}^{1,q}\left(  \mathbb{R}^{N}\right)  \cap
W_{loc}^{2,1}\left(  \mathbb{R}^{N}\right)  $ because $N>q^{\prime},$ and
$\tilde{U}$ is a classical solution in $\mathbb{R}^{N}\backslash\left\{
0\right\}  .$ Then it is a weak solution of $(P_{\mathbb{R}^{N},\infty}),$ and
$\tilde{U}\not \in C^{1}(Q_{\mathbb{R}^{N},\infty}).$ Since $\tilde{U}\in
C\left(  \mathbb{R}^{N}\right)  ,$ from Theorem \ref{exiloc}, or from
\cite{BeBALa}, there exists also a classical solution $U_{\tilde{C}}\in
C^{2,1}(Q_{\mathbb{R}^{N},\infty})$ of the problem, thus $U_{\tilde{C}}\neq
U_{0}.$

More generally, for any $C>0,$ there exists a classical solution $U_{C}$ with
trace $C\left\vert x\right\vert ^{\left\vert a\right\vert }.$ And $U_{C}$ is
obtained as the limit of the nondecreasing sequence of the unique solutions
$U_{n,C}$ with trace $\min(C\left\vert x\right\vert ^{\left\vert a\right\vert
},n),$ then it is radial. Moreover for any $\lambda>0,$ the function
$U_{n,C,\lambda}(x,t)=\lambda^{-a}U_{n,C}(\lambda x,\lambda^{2}t)$ admits the
trace $\min(C\left\vert x\right\vert ^{\left\vert a\right\vert },n\lambda
^{-a}).$ Therefore, denoting by $k_{\lambda,n}$ the integer part of
$n\lambda^{-a},$ there holds $U_{k_{\lambda,n},C}\leq U_{n,C,\lambda}\leq
U_{k_{\lambda,n}+1}$ from the comparison principle. And $U_{n,C,\lambda}(x,t)$
converges everywhere to $\lambda^{-a}U_{C}(\lambda x,\lambda^{2}t),$ thus
$U_{C}(x,t)=\lambda^{-a}U_{C}(\lambda x,\lambda^{2}t),$ that means $U_{C}$ is
self-similar. Then $U_{C}$ has the form (\ref{fom}), where $f\in C^{2}(\left[
0,\infty\right)  ),$ $f(0)\geqq0,f^{\prime}(0)=0,$  $\lim_{\eta\rightarrow
\infty}\eta^{-\left\vert a\right\vert /2}f(\eta)=C,$ and for any $\eta>0,$%
\begin{equation}
f^{\prime\prime}(\eta)+(\frac{N-1}{\eta}+\frac{\eta}{2})f^{\prime}(\eta
)-\frac{\left\vert a\right\vert }{2}f(\eta)-\left\vert f^{\prime}%
(\eta)\right\vert ^{q}=0.\label{tra}%
\end{equation}
From the Cauchy-Lipschitz Theorem, we find $f(0)>0,$ since $f\not \equiv 0,$
hence $f^{\prime\prime}(0)>0.$ The function $f$ is increasing: indeed if there
exists a first point $\eta_{0}>0$ such that $f^{\prime}(\eta_{0})=0,$ then
$f^{\prime\prime}(\eta_{0})>0,$ which is contradictory.
\end{proof}

\section{Second local regularizing effect\label{sec6}}

Here we show the second regularizing effect. We prove an estimate, playing the
role of the sub-caloricity estimate (\ref{nusc}). Our proof follows the
general scheme of Stampacchia's method, developped by many authors, see
\cite{DiBe1} and references there in, and \cite{FoSoVe}.\medskip

First we write estimate (\ref{ruc}) in another form, and from Gagliardo
estimate, we obtain the following:

\begin{lemma}
\label{xit}Let $q>1.$ Let $\eta>0,r\geqq1.$ Let $u$ be any nonnegative weak
subsolution of equation (\ref{un}) in $Q_{\Omega,T}.$ Let $B_{2\eta}%
\subset\subset\Omega,$ $0<\theta<\tau<T,$ and $\xi\in C^{1}((0,T),C_{c}%
^{1}(\Omega)),$ with values in $\left[  0,1\right]  ,$ such that $\xi(.,t)=0$
for $t\leqq\theta$. Let $\lambda\geqq\max(2,q^{\prime}).$ \medskip

Then for any $\nu\in\left(  0,1\right]  ,$
\begin{equation}
\sup_{\left[  \theta,\tau\right]  }\int_{\Omega}u^{r}(.,t)\xi^{\lambda}%
+\frac{\int_{\theta}^{\tau}\int_{\Omega}u^{(q+r-1)(1+\frac{\mu}{N})}%
\xi^{\lambda(1+\frac{\mu}{N})}}{(\sup_{t\in\left[  \theta,\tau\right]  }%
\int_{\Omega}u^{r}\xi^{\frac{\lambda r}{q+r-1}})^{\frac{q}{N}}}\leqq
C\int_{\theta}^{\tau}\int_{\Omega}(u^{r}\left\vert \xi_{t}\right\vert
+u^{r-1}\left\vert \nabla\xi\right\vert ^{q^{\prime}}+u^{q+r-1}\left\vert
\nabla\xi\right\vert ^{q}), \label{fle}%
\end{equation}
where $\mu=rq/(q+r-1),$ $C=C(N,q,r,\lambda).$
\end{lemma}

\begin{proof}
From Remark \ref{subreg}, $u\in L_{loc}^{\infty}(Q_{\Omega,T})),$ and hence
$u^{\frac{q+r-1}{q}}\xi^{\frac{\lambda}{q}}\in W^{1,q}(Q_{\Omega,\theta,t})$
and
\begin{align*}
\int_{\theta}^{t}\int_{\Omega}|\nabla(u^{\frac{q+r-1}{q}}\xi^{\frac{\lambda
}{q}})|^{q}  &  =\int_{\theta}^{t}\int_{\Omega}\left\vert \frac{q+r-1}%
{q}u^{\frac{r-1}{q}}\xi^{\frac{\lambda}{q}}\nabla u+\frac{\lambda}{q}%
u^{\frac{q+r-1}{q}}\xi^{\frac{\lambda-q}{q}}\nabla\xi\right\vert ^{q}\\
&  \leqq C(\int_{\theta}^{t}\int_{\Omega}u^{r-1}|\nabla u|^{q}\xi^{\lambda
}+\int_{\theta}^{t}\int_{\Omega}u^{q+r-1}|\nabla\xi|^{q}\xi^{\lambda-q}),
\end{align*}
with $C=C(q,r,\lambda).$ From (\ref{ruc}), since $\nu\leqq1,$ we get%
\begin{equation}
\sup_{\left[  \theta,\tau\right]  }\int_{\Omega}u^{r}(.,t)\xi^{\lambda}%
+\int_{\theta}^{\tau}\int_{\Omega}|\nabla(u^{\frac{q+r-1}{q}}\xi
^{\frac{\lambda}{q}})|^{q}\leqq C\int_{\theta}^{\tau}\int_{\Omega}%
(u^{r}\left\vert \xi_{t}\right\vert +u^{r-1}\left\vert \nabla\xi\right\vert
^{q^{\prime}}+u^{q+r-1}\left\vert \nabla\xi\right\vert ^{q}), \label{flo}%
\end{equation}
where $C=C(q,r,\lambda).$ Next we use a Galliardo type estimate, see
\cite[Proposition 3.1]{DiBe1}: for any $\mu\geqq1,$ and any $w\in
L_{loc}^{\infty}((0,T),L^{\mu}(\Omega))\cap L_{loc}^{q}((0,T),W^{1,q}%
(\Omega)),$
\[
\int_{\theta}^{\tau}\int_{\Omega}w^{q(1+\frac{\mu}{N})})\leqq C(\int_{\theta
}^{\tau}\int_{\Omega}|\nabla w|^{q})(\sup_{t\in\left[  \theta,\tau\right]
}\int_{\Omega}|w|^{\mu})^{\frac{q}{N}},\qquad C=C(N,q,\mu).
\]
Taking $w=u^{\frac{q+r-1}{q}}\xi^{\frac{\lambda}{q}}$ and $\mu=qr/(q+r-1)\geqq
r\geqq1,$ setting $s=1+\mu/N,$ it comes
\[
\int_{\theta}^{\tau}\int_{\Omega}u^{(q+r-1)s}\xi^{\lambda s}\leqq
C(\int_{\theta}^{\tau}\int_{\Omega}|\nabla w|^{q})(\sup_{t\in\left[
\theta,\tau\right]  }\int_{\Omega}u^{r}\xi^{\frac{\lambda r}{q+r-1}}%
)^{\frac{q}{N}},
\]
hence (\ref{fle}) follows.\medskip
\end{proof}

\begin{theorem}
\label{locest}Let $q>1.$ Let $u$ be any nonnegative weak solution of equation
(\ref{un}) in $Q_{\Omega,T}.$ Let $B(x_{0},\rho)\subset\subset\Omega.$ Let
$R>q-1$ (in particular any $R\geqq1$ if $q<2).$ Then there exists $C=C(N,q,R)$
such that, for any $t,\theta$ such that $0<t-2\theta<t<T,$
\begin{align}
\sup_{B(x_{0},\frac{\rho}{2})\times\left[  t-\theta,t\right]  }u  &  \leqq
C\theta^{-\frac{N+q}{qR+N(q-1)}}(\int_{t-2\theta}^{t}\int_{B(x_{0},\rho)}%
u^{R})^{\frac{q}{qR+N(q-1)}}\nonumber\\
&  +C\rho^{-\frac{N+q}{(q-1)(R+N+1)}}(\int_{t-2\theta}^{t}\int_{B(x_{0},\rho
)}u^{R})^{\frac{1}{R+N+1}}+C\rho^{-\frac{N+q}{R+1-q}}(\int_{t-2\theta}^{t}%
\int_{B(x_{0},\rho)}u^{R})^{\frac{1}{R+1-q}}. \label{gr2}%
\end{align}

\end{theorem}

\begin{proof}
Since $u\in C((0,T);L_{loc}^{R}(Q_{\Omega,T})),$ by regularization we can
assume that $u$ is a classical solution in $Q_{\Omega,T}.$ Let $t,\theta$ such
that $0<t-2\theta<t<T.$ We can assume $x_{0}=0\in\Omega.$ By translation of
$t-\theta,$ we are lead to prove that for any solution in $Q_{\Omega
,-\tau/2,\tau/2}\ $($\tau<T$),
\begin{align}
\sup_{Q_{B_{\rho/2}},0,\theta}u  &  \leqq C\theta^{-\frac{N+q}{qR+N(q-1)}%
}(\int_{-\theta}^{\theta}\int_{B_{\rho}}u^{R})^{\frac{q}{qR+N(q-1)}%
}\nonumber\\
&  +C\rho^{-\frac{N+q}{(q-1)(R+N+1)}}(\int_{-\theta}^{\theta}\int_{B_{\rho}%
}u^{R})^{\frac{1}{R+N+1}}+C\rho^{-\frac{N+q}{R+1-q}}(\int_{-\theta}^{\theta
}\int_{B_{\rho}}u^{R})^{\frac{1}{R+1-q}}. \label{grf}%
\end{align}
For given $k>0$ we set $u_{k}=(u-k)^{+}$ . Then $u_{k}\in C(0,T);L_{loc}%
^{R}(Q_{\Omega,T})),$ and $u_{k}$ is a weak subsolution of equation
(\ref{un}), from the Kato inequality. We set
\begin{align*}
\rho_{n}  &  =(1+2^{-n})\rho/2,\qquad t_{n}=-(1+2^{-n})\theta/2,\\
Q_{n}  &  =B_{\rho_{n}}\times(t_{n},\theta),\qquad Q_{0}=B_{\rho}%
\times(-\theta,\theta),\qquad Q_{\infty}=B_{\rho/2}\times(-\theta/2,\theta),\\
k_{n}  &  =(1-2^{-(n+1)})k,\text{\qquad}\tilde{k}=(k_{n}+k_{n+1})/2.
\end{align*}
and set $M_{\sigma}=\sup_{Q_{\infty}}u,$ $M=\sup_{Q_{0}}u.$ Let $\xi
(x,t)=\xi_{1}(x)\xi_{2}(t)$ where $\xi_{1}\in C_{c}^{1}(\Omega),$ $\xi_{2}\in
C^{1}(\mathbb{R)},$ with values in $\left[  0,1\right]  $, such that%
\begin{align*}
\xi_{1}  &  =1\quad\text{on }B_{\rho_{n+1}},\quad\quad\xi_{1}=0\quad\text{on
}\mathbb{R}^{N}\backslash B_{\rho_{n}},\quad\quad\left\vert \nabla\xi
_{1}\right\vert \leqq C(N)2^{n+1}/\rho;\\
\xi_{2}  &  =1\quad\text{on }\left[  \theta_{n+1},\infty\right)  ,\quad
\quad\xi_{2}=0\quad\text{on }\left(  -\infty,\theta_{n}\right]  ,\quad
\quad\left\vert \xi_{2,t}\right\vert \leqq C(N)2^{n+1}/\theta.
\end{align*}
From Lemma \ref{xit} we get, with $\mu=qr/(q+r-1),$
\begin{align*}
&  \sup_{t\in\left[  t_{n+1},\theta\right]  }%
{\textstyle\int_{B_{\rho_{n+1}}}}
u_{k_{n+1}}^{r}(.,t)+\frac{\int_{t_{n+1}}^{\theta}\int_{B_{\rho_{n+1}}%
}u_{k_{n+1}}^{(q+r-1)(1+\frac{\mu}{N})}}{(\sup_{t\in\left[  t_{n}%
,\theta\right]  }\int_{B_{\rho_{n}}}u_{k_{n}}^{r})^{\frac{q}{N}}}\leqq
CX_{n},\text{ where }\\
X_{n}  &  =\int_{t_{n}}^{\theta}\int_{B_{\rho_{n}}}(u_{k_{n+1}}^{r}\left\vert
\zeta_{t}\right\vert +u_{k_{n+1}}^{r-1}\left\vert \nabla\xi\right\vert
^{q^{\prime}}+u_{k_{n+1}}^{q+r-1}\left\vert \nabla\xi\right\vert ^{q})).
\end{align*}
Let us define
\[
Y_{n}=\int_{t_{n}}^{\theta}\int_{B_{\rho_{n}}}u_{k_{n}}^{q+r-1},\quad
Z_{n}=\sup_{t\in\left[  t_{n},\theta\right]  }\int_{B_{\rho_{n}}}u_{k_{n}}%
^{r},\quad W_{n}=\int_{t_{n}}^{\theta}\int_{B_{\rho_{n}}}\chi_{\left\{  u\geqq
k_{n}\right\}  }.
\]
Thus, from the H\"{o}lder inequality,
\begin{equation}
Z_{n+1}+Z_{n}^{-\frac{q}{N}}W_{n+1}^{-\frac{\mu}{N}}{}Y_{n+1}^{1+\frac{\mu}%
{N}}\leqq CX_{n}. \label{rele}%
\end{equation}
Morever, for any $\gamma,\beta>0,$
\begin{align*}
\int_{t_{n}}^{\theta}\int_{B_{\rho_{n}}}u_{k_{n}}^{\gamma+\beta}  &  \geqq
\int_{t_{n}}^{\theta}\int_{B_{\rho_{n}}}(k_{n}-k_{n+1})^{\gamma+\beta}%
\chi_{\left\{  u\geqq k_{n+1}\right\}  }\\
&  \geqq(k2^{-(n+2)})^{\gamma+\beta}\int_{t_{n}}^{\theta}\int_{B_{\rho_{n}}%
}\chi_{\left\{  u\geqq k_{n+1}\right\}  }\geqq(k2^{-(n+2)})^{\gamma+\beta}%
\int_{t_{n+1}}^{\theta}\int_{B_{\rho_{n+1}}}\chi_{\left\{  u\geqq
k_{n+1}\right\}  },
\end{align*}
and from the H\"{o}lder inequality,
\begin{align*}
\int_{t_{n}}^{\theta}\int_{B_{\rho_{n}}}u_{k_{n+1}}^{\gamma}  &  \leqq
(\int_{t_{n}}^{\theta}\int_{B_{\rho_{n}}}u_{k_{n+1}}^{\gamma+\beta}%
)^{\frac{\gamma}{\gamma+\beta}}(\int_{t_{n}}^{\theta}\int_{B_{\rho_{n}}}%
\chi_{\left\{  u\geqq k_{n+1}\right\}  })^{\frac{\beta}{\gamma+\beta}}\\
&  \leqq(\int_{t_{n}}^{\theta}\int_{B_{\rho_{n}}}u_{k_{n}}^{\gamma+\beta
})(k^{-1}2^{(n+2)})^{\beta}(\int_{t_{n}}^{\theta}\int_{B_{\rho_{n}}}u_{k_{n}%
}^{\gamma+\beta})^{\frac{\beta}{\gamma+\beta}}\\
&  \leq(k^{-1}2^{(n+2)})^{\beta}\int_{t_{n}}^{\theta}\int_{B_{\rho_{n}}%
}u_{k_{n}}^{\gamma+\beta}.
\end{align*}
Thus in particular
\begin{equation}
W_{n+1}\leqq C(\frac{2^{n+1}}{k})^{q+r-1}Y_{n},\qquad\int_{t_{n}}^{\theta}%
\int_{B_{\rho_{n}}}u_{k_{n+1}}^{r}\leqq C(\frac{2^{n+1}}{k})^{q-1}Y_{n}%
,\qquad\int_{t_{n}}^{\theta}\int_{B_{\rho_{n}}}u_{k_{n+1}}^{r-1}\leqq
C(\frac{2^{n+1}}{k})^{q}Y_{n}. \label{rela}%
\end{equation}
Otherwise
\[
X_{n}\leqq\int_{t_{n}}^{\theta}\int_{B_{\rho_{n}}}(2^{n+1}\theta
^{-1}u_{k_{n+1}}^{r}+2^{q^{\prime}(n+1)}\rho^{-q^{\prime}}u_{k_{n+1}}%
^{r-1}+2^{q(n+1)}\rho^{-q}u_{k_{n+1}}^{q+r-1}),
\]
then from (\ref{rela}),
\begin{equation}
X_{n}\leqq Cb_{0}^{n}f(\theta,\rho,k)Y_{n},\qquad\text{where }f(\theta
,\rho,k)=(\theta^{-1}\frac{1}{k^{q-1}}+\frac{1}{k^{q}}\rho^{-q^{\prime}}%
+\rho^{-q}). \label{reli}%
\end{equation}
for some $b_{0}$ depending on $q,r.$ Then from (\ref{rele}), (\ref{rela}) and
(\ref{reli}),
\[
Z_{n+1}\leqq Cb_{0}^{n}f(\theta,\rho,k)Y_{n},\qquad Y_{n+1}^{1+\frac{\mu}{N}%
}\leqq CZ_{n}^{\frac{q}{N}}(\frac{2^{n+1}}{k})^{(q+r-1)\frac{\mu}{N}}b_{0}%
^{n}f(\theta,\rho,k)Y_{n}^{1+\frac{\mu}{N}}.
\]
Since $Y_{n+1}\leqq Y_{n},$setting $\alpha=q/(N+\mu)$ and denoting by
$b_{1},b$ some new constants depending on $N,q,r,$
\begin{align*}
Y_{n+2}  &  \leqq CZ_{n+1}^{\frac{q}{N+\mu}}b_{1}^{n+1}k^{-(q+r-1)\frac{\mu
}{N+\mu}}f^{\frac{N}{N+\mu}}(\theta,\rho,k)Y_{n+1}\\
&  \leqq C(b_{0}^{n}f(\theta,\rho,k)Y_{n})^{\frac{q}{N+\mu}}b_{1}%
^{n+1}k^{-(q+r-1)\frac{\mu}{N+\mu}}f^{\frac{N}{N+\mu}}(\theta,\rho,k)Y_{n}\\
&  \leqq Cb^{n}f^{\frac{N+q}{N+\mu}}k^{-(q+r-1)\frac{\mu}{N+\mu}}%
Y_{n}^{1+\frac{q}{N+\mu}}:=Db^{n}Y_{n}^{1+\alpha}.
\end{align*}
From \cite[Lemma 4.1]{DiBe1}, $Y_{n}\rightarrow0$ if
\[
Y_{0}^{\alpha}\delta^{1/\alpha}\leqq D^{-1}=C^{-1}k^{(q+r-1)\frac{\mu}{N+\mu}%
}f^{-\frac{N+q}{N+\mu}},
\]
that means
\begin{equation}
k^{qr}\geqq cY_{0}^{q}((\theta^{-1}\frac{1}{k^{q-1}}+\frac{1}{k^{q}}%
\rho^{-q^{\prime}}+\rho^{-q}))^{N+q}. \label{res}%
\end{equation}
For getting (\ref{res}) it is sufficient that
\[
k^{qr+(q-1)(N+q)}\geqq\frac{c}{2}Y_{0}^{q}\theta^{-(N+q)},\text{\quad
}k^{(r+N+q)}\geqq(\frac{c}{2})^{1/q}Y_{0}\rho^{-\frac{N+q}{q-1}},\text{ \quad
and }k^{r}\geqq\frac{c}{2}Y_{0}\rho^{-(N+q)}.
\]
Thus we deduce that
\begin{align}
\sup_{Q_{\infty}}u  &  \leqq C\theta^{-\frac{N+q}{qr+(N+q)(q-1)}}%
(\int_{-\theta}^{\theta}\int_{B_{\rho}}u^{q+r-1})^{\frac{q}{qr+(N+q)(q-1)}%
}\nonumber\\
&  +C\rho^{-\frac{N+q}{(q-1)(r+N+q)}}(\int_{-\theta}^{\theta}\int_{B_{\rho}%
}u^{q+r-1})^{\frac{1}{r+N+q}}+C\rho^{-\frac{N+q}{r}}(\int_{-\theta}^{\theta
}\int_{B_{\rho}}u^{q+r-1})^{\frac{1}{r}}. \label{vol}%
\end{align}
If we set $q+r-1=R$, we obtain (\ref{grf}) for any $R\geqq q$.$\medskip$

Next we consider the case $R<q.$ From (\ref{vol}) we get
\begin{align*}
\sup_{B_{\sigma\rho}\times(-\theta/2,\theta)}u  &  \leqq C\theta^{-\frac
{N+q}{q+(q-1)(N+q)}}(\int_{0}^{\theta}\int_{B_{\rho}}u^{q})^{\frac
{q}{q+(q-1)(N+q)}}\\
&  +C\rho^{-\frac{N+q}{(q-1)(1+N+q)}}(\int_{-\theta}^{\theta}\int_{B_{\rho}%
}u^{q})^{\frac{1}{1+N+q}}+C\rho^{-(N+q)}\int_{-\theta}^{\theta}\int_{B_{\rho}%
}u^{q}\\
&  \leqq C\theta^{-\frac{N+q}{q+(q-1)(N+q)}}(\sup_{B_{\rho}\times0,\theta
)}u)^{\frac{q(q-R)}{q+(q-1)(N+q)}}(\int_{-\theta}^{\theta}\int_{B_{\rho}}%
u^{R})^{\frac{q}{q+(q-1)(N+q)}}\\
&  +C\rho^{-\frac{N+q}{(q-1)(1+N+q)}}(\sup_{B_{\rho}\times0,\theta)}%
u)^{\frac{q(q-R)}{1+N+q)}}(\int_{-\theta}^{\theta}\int_{B_{\rho}}u^{R}%
)^{\frac{1}{1+N+q}}\\
&  +C\rho^{-(N+q)}(\sup_{B_{\rho}\times0,\theta)}u)^{(q-R)}\int_{-\theta
}^{\theta}\int_{B_{\rho}}u^{R}.
\end{align*}
We define
\[
\tilde{\rho}_{n}=(1+2^{-(n+1)})\rho,\qquad\theta_{n}=-(1+2^{-(n+1)}%
)\theta,\qquad\tilde{Q}_{n}=B_{\tilde{\rho}_{n}}\times(\theta_{n}%
,\theta),\qquad M_{n}=\sup_{\tilde{Q}_{n}}u,
\]
hence $M_{0}=\sup_{B_{\rho/2}\times(-\theta/2,\theta)}u.$ We find
\begin{align*}
M_{n}  &  \leqq C\theta^{-\frac{N+q}{q+(q-1)(N+q)}}M_{n+1}^{\frac
{q(q-R)}{q+(q-1)(N+q)}}(\int_{-\theta}^{\theta}\int_{B_{\rho}}u^{R})^{\frac
{q}{q+(q-1)(N+q)}}\\
&  +C\rho^{-\frac{N+q}{(q-1)(1+N+q)}}M_{n+1}^{\frac{q(q-R)}{1+N+q}}%
(\int_{-\theta}^{\theta}\int_{B_{\rho}}u^{R})^{\frac{1}{1+N+q}}+C\rho
^{-(N+q)}M_{n+1}^{q-R}\int_{-\theta}^{\theta}\int_{B_{\rho}}u^{R}.
\end{align*}
We set
\begin{align*}
I  &  =C\theta^{-\frac{N+q}{q+(q-1)(N+q)}}(\int_{-\theta}^{\theta}%
\int_{B_{\rho}}u^{R})^{\frac{q}{q+(q-1)(N+q)}},\\
J  &  =C\rho^{-(N+q)}\int_{0}^{\theta}\int_{B_{\rho}}u^{R},\qquad
L=C\rho^{-\frac{N+q}{(q-1)(1+N+q)}}(\int_{-\theta}^{\theta}\int_{B_{\rho}%
}u^{R})^{\frac{1}{1+N+q}}.
\end{align*}
Note that $R>q-1,$ that means $q-R<1.$ Then from H\"{o}lder inequality,
\[
M_{n}\leqq\frac{1}{2}M_{n+1}+C(I^{\sigma}+L^{\delta}+J^{\frac{1}{R+1-q}%
}),\qquad\sigma=\frac{q+(q-1)(N+q)}{N(q-1)+qR},\quad\delta=\frac{1+N+q}%
{R+N+1}.
\]
Thus $M_{0}\leqq2^{-n}M_{n}+2C(I^{\sigma}+L^{\delta}+J^{\frac{1}{R+1-q}}),$
and finally
\begin{align*}
M_{0}  &  =\sup_{Q_{0}}u\leqq C(I^{\sigma}+L^{\delta}+J^{\frac{1}{R+1-q}%
})=C\theta^{-\frac{N+q}{N(q-1)+qR}}(\int_{-\theta}^{\theta}\int_{B_{\rho}%
}u^{R})^{\frac{q}{N(q-1)+qR}}\\
&  +C\rho^{-\frac{N+q}{(q-1)(R+N+1)}}(\int_{-\theta}^{\theta}\int_{B_{\rho}%
}u^{R})^{\frac{1}{^{R+N+1}}}+C\rho^{-\frac{N+q}{R+1-q}}(\int_{-\theta}%
^{\theta}\int_{B_{\rho}}u^{R})^{\frac{1}{^{R+1-q}}},
\end{align*}
which shows again (\ref{grf}). Then (\ref{grf}) holds for any $R>q-1,$ in
particular for any $R\geqq1$ if $q<2.$\medskip
\end{proof}

Now we prove our second regularing effect due to the effect of the gradient:

\begin{proof}
[Proof of Theorem \ref{effects}]We assume $x_{0}=0.$ Let $\kappa>0$ be a
parameter. From (\ref{gr2}), for any $\rho\in(0,\eta)$ such that $\rho
^{\kappa}\leqq t<\tau,$%
\begin{align*}
\sup_{B_{\frac{\rho}{2}}\times\left[  t-\rho^{\kappa},t\right]  }u  &  \leqq
C\rho^{-\frac{\kappa(N+q)}{qR+N(q-1)}}(\int_{t-\rho^{\kappa}}^{t}\int%
_{B_{\rho}}u^{R})^{\frac{q}{qR+N(q-1)}}\\
&  +C\rho^{-\frac{N+q}{(q-1)(R+N+1)}}(\int_{t-\rho^{\kappa}}^{t}\int_{B_{\rho
}}u^{R})^{\frac{1}{R+N+1}}+C\rho^{-\frac{N+q}{R+1-q}}(\int_{t-\rho^{\kappa}%
}^{t}\int_{B_{\rho}}u^{R})^{\frac{1}{R+1-q}},
\end{align*}
where $C=C(N,q,R).$ Now from estimate (\ref{locint}) of Lemma \ref{cor},
\begin{align*}
\text{sup}_{B_{\eta/2}}u(.,t)  &  \leqq C\rho^{-\frac{\kappa N}{qR+N(q-1)}%
}(\eta^{\frac{N}{R}-q^{\prime}}t+\left\Vert u_{0}\right\Vert _{L^{R}(B_{\eta
})})^{\frac{Rq}{qR+N(q-1)}}\\
&  +C\rho^{-\frac{N+q}{(q-1)(R+N+1)}+\frac{\kappa}{R+N+1}}(\eta^{\frac{N}%
{R}-q^{\prime}}t+\left\Vert u_{0}\right\Vert _{L^{R}(B_{\eta})})^{\frac
{R}{R+N+1}}\\
&  +C\rho^{\frac{-(N+q)+\kappa}{R+1-q}}(\eta^{\frac{N}{R}-q^{\prime}%
}t+\left\Vert u_{0}\right\Vert _{L^{R}(B_{\eta})})^{\frac{R}{R+1-q}}.
\end{align*}
Let $\tau<T,$ and $k_{0}\in\mathbb{N}$ such that $k_{0}\eta^{\kappa}%
/2\geqq\tau.$ For any $t\in\left(  0,\tau\right]  ,$ there exists
$k\in\mathbb{N}$ with $k\leqq k_{0}$ such that $t\in\left(  k\eta^{\kappa
}/2,(k+1)\eta^{\kappa}/2\right]  .$ taking $\rho^{\kappa}=t/(k+1),$ we find
for any $0<t<\tau,$ and $C=C(N,q,R),$
\begin{align}
\text{sup}_{B_{\eta/2}}u(.,t)  &  \leqq C(\frac{1+\eta^{-\kappa}\tau}%
{t})^{\frac{N}{qR+N(q-1)}}(\eta^{\frac{N}{R}-q^{\prime}}t+\left\Vert
u_{0}\right\Vert _{L^{R}(B_{\eta})})^{\frac{Rq}{qR+N(q-1)}}\nonumber\\
&  +C(\frac{1+\eta^{-\kappa}\tau}{t})^{\frac{\frac{N+q}{\kappa(q-1)}-1}%
{R+N+1}}(\eta^{\frac{N}{R}-q^{\prime}}t+\left\Vert u_{0}\right\Vert
_{L^{R}(B_{\eta})})^{\frac{R}{R+N+1}}\nonumber\\
&  +C(\frac{1+\eta^{-\kappa}\tau}{t})^{\frac{\frac{N+q}{\kappa}-1}{R+1-q}%
}(\eta^{\frac{N}{R}-q^{\prime}}t+\left\Vert u_{0}\right\Vert _{L^{R}(B_{\eta
})})^{\frac{R}{R+1-q}}. \label{art}%
\end{align}
If we choose $\kappa$ such that $\kappa\varepsilon(N+q)q^{\prime}\geqq1,$ we
obtain, with $C=C(N,q,R,\eta,\varepsilon,\tau),$
\begin{align}
\text{sup}_{B_{\eta/2}}u(.,t)  &  \leqq Ct^{-\frac{N}{qR+N(q-1)}}(t+\left\Vert
u_{0}\right\Vert _{L^{R}(B_{\eta})})^{\frac{Rq}{qR+N(q-1)}}\nonumber\\
&  +Ct^{\frac{1-\varepsilon}{R+N+1}}(t+\left\Vert u_{0}\right\Vert
_{L^{R}(B_{\eta})})^{\frac{R}{R+N+1}}+Ct^{\frac{1-\varepsilon}{R+1-q}%
}(t+\left\Vert u_{0}\right\Vert _{L^{R}(B_{\eta})})^{\frac{R}{R+1-q}}
\label{gin}%
\end{align}
And in fact the second term can be absorbed by the first one, with a new
constant depending on $\tau,$ and we finally obtain (\ref{epsi}).\medskip
\end{proof}

\begin{remark}
These estimate in $t^{-N/(qR+N(q-1))}$ improves the estimate in $t^{-N/2R}$of
the first regularizing effect when $q>q_{\ast}$. And it appears to be sharp.
Indeed consider for example the particular solutions given in \cite{QW} of the
form $u_{C}(x,t)=Ct^{-a/2}f(\left\vert x\right\vert /\sqrt{t}),$ where
$\eta\mapsto f(\eta)$ is bounded, $f^{\prime}(0)=0$ and $\lim_{\eta
\rightarrow\infty}\eta^{a}f\left(  \eta\right)  =C.$ Then $u_{C}$ is solution
of (\ref{un}) in $Q_{\mathbb{R}^{N}\backslash\left\{  0\right\}  ,\infty},$
with initial data $C\left\vert x\right\vert ^{-a}.$ When $a<N,$ that means
$q>q_{\ast}$, then $\left\vert x\right\vert ^{-a}\in L_{loc}^{R}%
(\mathbb{R}^{N})$ for any $R\in\left[  1,N/a\right)  $, and $u_{C}$ is
solution in $Q_{\mathbb{R}^{N},\infty}.$ We have $\sup_{B_{1}}%
u(.,t)=Cf(0)t^{-a/2}.$ Taking $N/R=a(1+\delta),$ for small $\delta>0$ our
estimate near $t=0$ gives $\sup_{B_{1}}u(.,t)\leqq C_{\delta}t^{-\frac{a}%
{2}(1+\delta)}.$
\end{remark}


\begin{thebibliography}{99}                                                                                               %


\bibitem {Al}N. Alaa, \textit{Solutions faibles d'\'{e}quations paraboliques
quasiln\'{e}aires avec donn\'{e}es initiales mesures}, Ann. Math. Blaise
Pascal, 3 (1996), 1-15.

\bibitem {AmBA}L. Amour and M. Ben-Artzi, \textit{Global existence and decay
for Viscous H amilton-Jacobi equations,} Nonlinear Anal., Methods and Appl.,
31 (1998), 621-628.

\bibitem {BaPi}P. Baras and M. Pierre, \textit{Probl\`{e}mes paraboliques
semi-lin\'{e}aires avec donn\'{e}es mesures,} Applicable Anal., 18 (1984), 111-149.

\bibitem {BarLa}J. Bartier and P. Lauren\c{c}ot, \textit{Gradient estimates
for a degenerate parabolic equation with gradient absorption and
applications,} J. Funct. Anal. 254 (208), 851-878.

\bibitem {BeBALa}S. Benachour, M. Ben Artzi, and P. Lauren\c{c}ot,
\textit{Sharp decay estimates and vanishing viscosity for diffusive
Hamilton-Jacobi equations,} Adv. Diff. Equ., 14 (2009), no. 1-2, 1--25.

\bibitem {BeDa}S. Benachour and S. Dabuleanu, \textit{The mixed
Cauchy-Dirichlet problem for a viscous Hamilton-Jacobi equation,} Advances
Diff. Equ., 8 (2003), 1409-1452.

\bibitem {BeKaLa}S. Benachour, G. Karch and P. Lauren\c{c}ot,
\textit{Asymptotic profiles of solutions to viscous Hamilton-Jacobi
equations}, J. Math. Pures Appl., 83 (2004), 1275-1308.

\bibitem {BeKoLa}S. Benachour, H.Koch, and P. Lauren\c{c}ot, \textit{Very
singular solutions to a nonlinear parabolic equation with absorption, }II-
Uniqueness, Proc. Roy. Soc. Edinburgh Sect. A 134 (2004), 39-54.

\bibitem {BeLa99}S. Benachour and P. Lauren\c{c}ot, \textit{Global solutions
to viscous Hamilton-Jacobi equations with irregular initial data,} Comm.
Partial Diff. Equ., 24 (1999), 1999-2021.

\bibitem {BeLa01}S. Benachour and P. Lauren\c{c}ot, \textit{Very singular
solutions to a nonlinear parabolic equation with absorption, I- Existence,
}Proc. Roy. Soc. Edinburgh Sect. A, 131 (2001), 27-44.

\bibitem {BASoWe}M. Ben Artzi, P. Souplet and F. Weissler, \textit{The local
theory for Viscous Hamilton-Jacobi equations in Lebesgue spaces}, J. Math.
Pures Appl., 81 (2002), 343-378.

\bibitem {BiDao1}M.F. Bidaut-V\'{e}ron, and A.N. Dao, \textit{Isolated initial
singularities for the viscous Hamilton Jacobi equation}, Advances in Diff.
Equations, 17 (2012), 903-934.

\bibitem {BiDao2}M.F. Bidaut-V\'{e}ron, and A.N. Dao, \textit{L}$^{\infty}%
$\textit{ estimates and uniqueness results for nonlinear parabolic equations
with gradient absorption terms}, Nonlinear Analysis, 91 (2013), 121-152.

\bibitem {BiDao3}M.F. Bidaut-V\'{e}ron, and A.N. Dao, \textit{Initial trace of
solutions of Hamilton-Jacobi equation with absorption, }preprint.

\bibitem {BiGuKa}P. Biler, M. Guedda and G. Karch, \textit{Asymptotic
properties of solutions of the viscous Hamilton-Jacobi equation,} J. Evol.
Equ. 4 (2004),

\bibitem {CLS}M. Crandall, P.\ Lions and P. Souganidis, \textit{Maximal
solutions and universal bounds for some partial differential equations of
evolution, }Arch. Rat. Mech. Anal. 105 (1989), 163-190.

\bibitem {DiBe1}E. Di Benedetto, Degenerate parabolic equations, Springer
Verlag (1993).

\bibitem {DiBe}E. Di Benedetto, Partial Differential Equations, Birkhauser,
2nd ed., Boston, Basel, Berlin (2010).

\bibitem {FoSoVe}S. Fornaro, M. Sosio and V.\ Vespri, $L_{loc}^{r}$%
-$L_{loc}^{\infty}$ \textit{estimates and expansion of positivity for a class
of doubly nonlinear singular parabolic equations}, Discrete Cont. Dyn.
Systems, 7 (2014), 737-760.

\bibitem {GiGuKe}B. Gilding, M. Guedda and R. Kersner, \textit{The Cauchy
problem for} $u_{t}=\Delta u+\left\vert \nabla u\right\vert ^{q},$ J. Math.
Anal. Appl. 284 (2003), 733-755.

\bibitem {IaLa}R. Iagar and P. Lauren\c{c}ot, \textit{Positivity, decay, and
extinction for a singular diffusion equation with gradient absorption}, J.
Funct. Anal., 262 (2012), 3186--3239.

\bibitem {LePe}L. Leonori and T. Petitta, \textit{Local estimates for
parabolic equations with nonlinear gradient terms}, Calc. Var. Part. Diff.
Equ.,42 (2011), 153-187.

\bibitem {Li}P.L. Lions, \textit{Regularizing effects for first-order
Hamilton-Jacobi equations,} Applicable Anal. 20 (1985), 283--307.

\bibitem {Po}A. Porretta, \textit{Existence results for nonlinear parabolic
equations via strong convergence of trucations, }Ann. Mat. Pura Appl., 177
(1999), 143-172.

\bibitem {QW}Y. Qi and M. Wang, \textit{The self-similar profiles of
generalized KPZ equation,} Pacific J. Math. 201 (2001), 223-240.

\bibitem {SoZh}P. Souplet and Q. Zhang, \textit{Global solutions of
inhomogeneous Hamilton-Jacobi equations,} J. Anal. Math. 99 (2006), 355-396.
\end{thebibliography}
\end{document}